\documentclass[12pt,reqno]{amsart}
\usepackage{hyperref} 
\usepackage{xcolor}
\newtheorem{thm}{Theorem}[section]
\newtheorem{prop}[thm]{Proposition}
\newtheorem{lem}[thm]{Lemma}
\newtheorem{corol}[thm]{Corollary}
\theoremstyle{definition}
\newtheorem{defin}[thm]{Definition}
\newtheorem{exa}[thm]{Example}

\newcommand{\forma}[1]{\langle #1 \rangle}
\newcommand{\norma}[1]{\Vert #1 \Vert}
\newcommand{\n}{\nabla}
\newcommand{\np}{\nabla^\perp}
\newcommand{\h}{\overrightarrow{H}}
\newcommand{\on}{\overline{\n}}
\title[]{Harmonic Gauss maps of submanifolds of arbitrary codimension of the Euclidean space and sphere and some applications}

\author{Daniel Bustos}
\address{Departamento de matem\'aticas y estad\'istica, Universidad del Tolima, Ibagué, Colombia}
\email{danielfbustosrios@gmail.com}
\author{Jaime B. Ripoll}
\address{Instituto de matem\'atica e estad\'istica, Universidade Federal do Rio Grande do Sul, Porto Alegre, Brazil}
\email{jaime.ripoll@ufrgs.br}

\begin{document}

\allowdisplaybreaks

\begin{abstract}
It is  proved results about existence and non existence of unit normal sections of submanifolds of the  Euclidean space and sphere which associated Gauss maps are harmonic. Some applications to constant mean curvature hypersurfaces of the sphere and to isoparametric submanifolds are obtained too.
\end{abstract}

\maketitle

%%%%%%%%%%%%%%%%%%%%%%%%%%%%%%%%%%%%%%%%%%%%%%%%%%%%%%%%%%%%%%%%%%%%%%%%%%%%%%%%%%%%%%%%%%%%%%%%%%%%%%%%%%%%%%%%%%%%%%%%
%%%%%%%%%%%%%%%%%%%%%%%%%%%%%%%%%%%%%%%%%%%%%%%%%%%%%%%%%%%%%%%%%%%%%%%%%%%%%%%%%%%%%%%%%%%%%%%%%%%%%%%%%%%%%%%%%%%%%%%%
%%%%%%%%%%%%%%%%%%%%%%%%%%%%%%%%%%%%%%%%%%%%%%%%%%%%%%%%%%%%%%%%%%%%%%%%%%%%%%%%%%%%%%%%%%%%%%%%%%%%%%%%%%%%%%%%%%%%%%%
\section{Introduction}
 
A well known theorem of Ruh-Vilms \cite{RV} establishes that the Grassmanian Gauss map 
of an orientable submanifold $M^n$ of $\mathbb{R}^{m},$ $1\leq n \leq m-1,$ is harmonic if and only
if the mean curvature vector of $M$ is parallel in the normal connection. In
codimension one this result is equivalent to say that the Gauss map
\begin{equation}
\begin{array}{rcl}
\gamma_{\eta}:M &\rightarrow & \mathbb{S}^{m-1}\\p &\longmapsto &\eta\left(  p\right)
\end{array}
\end{equation}
associated to a unit normal section $\eta$ of $M$ is harmonic if and only if
$M$ has constant mean curvature (CMC).  This relation in the codimension one case provides a strong tool to the study of CMC hypersurfaces of the Euclidean space, in particular their Gauss image, a classical topic of research in Differential Geometry (see, for instance,   \cite {H}). In this paper we try to extend the codimension
$1$ case of Ruh-Vilms theorem to submanifolds $M$ of arbitrary codimension
by finding geometric conditions on $M$ that guarantee the existence of a
unit normal section $\eta$ of $M$ such that the associated Gauss map
$\gamma_{\eta}:M\rightarrow\mathbb{S}^{m-1}$ is harmonic. 

Considering that the space of unit normal sections has locally higher dimension in codimension bigger than or equal to $2,$ we should expect the existence of unit normal sections with harmonic Gauss maps under usual geometric assumptions as minimality or parallelism of the mean curvature vector. However, this first intuition is not true in general. Indeed, in Theorem~\ref{Classification} below we prove that the existence of a unit normal section of a surface of $\mathbb{R}^4$  determining a harmonic Gauss map yields strong restrictions on the geometry of the surface. The  result is local.

\begin{thm}
\label{Classification} Let $M$ be a surface of $\mathbb{R}^{4}$ with
parallel mean curvature vector field such that the second fundamental form of $M$ spans the normal space of $M$ in $\mathbb{R}^{4}$ at each point. Then  any unit normal section of $M$ in $\mathbb{R}^4$ can be written as $\eta=a\nu+b\mu,$ where $\nu$ is a unit normal section of $M$ in $\mathbb{R}^4$ tangent to the unit sphere $\mathbb{S}^3,$ $\mu$ is a unit normal section of $M$ orthogonal to $\nu$ and $a,b$ are functions on $M$ with $a^{2}+b^{2}=1.$  And, up to isometries of $\mathbb{R}^4$ 
\begin{enumerate}
\item\label{abconst}  $\gamma_\eta$ is a harmonic map and  $a, b$ are constants if and only if one of the following alternatives is satisfied:
\begin{enumerate}
\item\label{class1a} $M$ is an open subset of the
product of circles $$\mathbb{S}^{1}(r)\times \mathbb{S}^{1}(\sqrt{1-r^2})\subset \mathbb{S}^{3},$$with $0<r<1.$

\item\label{class1b} $M$ is an open subset of a totally umbilical non geodesic surface of $\mathbb{S}^3.$

\item\label{class1c} $M$ is a minimal surface of $%
\mathbb{S}^{3}$ and up to a reorientation of $M$ and up to a composition of $M$ with the antipodal map, $\eta$ is the position vector of $M$ in $\mathbb{R}^4$ or $\eta$ is a unit normal section of $M$ in $\mathbb{S}^3.$
\end{enumerate}

\item\label{abno} If $\gamma_\eta$ is a harmonic map and  $a, b$ are not constants then  $M$ is a CMC surface of  $\mathbb{S}^3$ with a  non constant principal curvature  which is constant along a principal direction.
\end{enumerate}
\end{thm}

Concerning alternative \eqref{abno} of Theorem \ref{Classification}, CMC surfaces invariant by a one parameter subgroup of isometries of $\mathbb{S}^3,$ apart spheres and tori, have a non constant principal curvature  which is constant along a principal direction (see \cite{HL} for the minimal case). We do not know if the these surfaces admit a unit normal section which Gauss map is harmonic.

As a direct consequence of  Theorem \ref{Classification} it follows a  non existence result of harmonic Gauss maps.

\begin{corol}\label{cornonh}
Let $M$ be a minimal surface of $\mathbb{R}^4$ such that the second fundamental form of  $M$ spans the normal space of $M$ at each point. Then for any unit normal section $\eta$ the Gauss map 
\begin{eqnarray*}
\gamma_\eta: M &\to & \mathbb{S}^3\\
p &\mapsto & \eta(p)
\end{eqnarray*}
is not a harmonic map.
\end{corol}
 We also get an analogous result of Corollary \ref{cornonh} of non existence harmonic Gauss maps for minimal surfaces of $\mathbb{S}^4.$

\begin{thm}\label{nhS4}
Let $M$ be a minimal surface of $\mathbb{S}^4\subset \mathbb{R}^5$ such that the second fundamental form of $M$ spans the normal space of $M$ in $\mathbb{S}^4$ at each point. Then, for any unit normal section $\eta$ of $M$ in $\mathbb{S}^4$ the Gauss map  
\begin{eqnarray*}
\gamma_\eta:M &\to & \mathbb{S}^4\\
p&\mapsto & \eta(p)
\end{eqnarray*} 
is not a harmonic map.
\end{thm}

It seems that what is behind Ruh-Vilms result is a fact which is trivial in codimension 1, namely: unit normal sections are parallel in the normal connection. To bring up this relation in arbitrary codimension, we recall some preliminary facts which are introduced in a more general Riemannian setting for later use.

Let $M$ be a manifold immersed in a Rie\-mannian manifold $N.$ Since there is no possibility of confusion in our context, we identify, as usual, $M$ with its image on $N$ by the immersion. Denote by $S(M)$ the vector bundle over $M$ of symmetric linear transformations (s.l.t), that is, 
 $$S(M) = \{(p,T)\mid p\in M\text{ and } T : T_pM\to T_p M\text{ is a s.l.t}\}.$$
Let $\mathcal{N}(M)$ and  $\mathcal{S}(M)$   be the vector bundles over $M$ of the sections of the normal bundle $TM^\bot$ and the sections of $S(M),$ respectively. 

We denote by $\tilde{\mathcal{B}}:=\mathcal{B}^\ast \mathcal{B}$ the  Simons operator \cite{JS}, a section of the vector bundle ${\rm Hom}(T M^\perp, T M^\perp),$  where 
$ \mathcal{B}:\mathcal{N}(M)\to  \mathcal{S}(M)$ is the vector bundle  homomorphism
 $\mathcal B(\eta)=S_\eta, $ $S_\eta$ is the second fundamental form associated to $\eta,$ and $\mathcal{B}^\ast: \mathcal{S}(M)\to \mathcal{N}(M)$  is the fiber-wise adjoint map of $\mathcal B$ (considering in  $\mathcal{S}(M)$  the  Hilbert-Schmidt metric. See Section \ref{P} ahead for more details). 
 
Recall that a map from a manifold $M$ to $\mathbb{S}^{m-1}$
is harmonic, in the case that $M$ is compact, if it is a critical point of the functional $\mathcal{H}:C^{\infty
}\left(  M,\mathbb{S}^{m-1}\right)  \rightarrow\mathbb{R}$
\[
\mathcal{H}(f)=\int_{M}\norma{df} ^{2},\text{ }f\in C^{\infty
}\left(  M,\mathbb{S}^{m-1}\right).  
\]
If $M$ is only complete, a map is harmonic if it is a critical point of $\mathcal H$ restricted to compact subdomains of $M.$

Denote by $\mathcal{N}_{1}\left(  M\right)  $ the  subbundle  of unit normal
sections of $\mathcal{N}\left(  M\right).$ In the case that $M$ is compact, we say that a unit normal section
is harmonic if it is a critical point of the functional $\mathcal{N}%
:\mathcal{N}_{1}\left(  M\right)  \rightarrow\mathbb{R}$
\[
\mathcal{N}(\eta)=\int_{M}\norma{\nabla\eta} ^{2},\text{ }\eta
\in\mathcal{N}_{1}\left(  M\right),  
\]
where $\n$ is the Riemannian connection of $N.$ In the case that  $M$ is only complete, a unit normal section
is harmonic if it is a  critical point of $\mathcal N$ restricted to compact subdomains of $M.$  

Finally, recalling that a normal section $\eta$ of $M$ is parallel (in the normal conection) if $\left (\nabla_E\eta\right )^\bot=0,$ where $E$ is any tangent vector field of $M$ and $\bot$ is the orthogonal projection of $TN$ on $TM^\bot,$ we prove:

\begin{thm}\label{trn} Let $M$ be a smooth manifold immersed in $\mathbb{R}^{m}$ with
parallel mean curvature vector. Let $\eta$ be a parallel unit normal section
of $M.$ Then the following alternatives are equivalent:
\begin{enumerate}
\item\label{trn1} $\gamma_{\eta}\in C^{\infty}\left(  M,\mathbb{S}^{m-1}\right)  $ is harmonic.

\item\label{trn2} $\eta$ is a harmonic section.

\item\label{trn3} $\eta$ is an eigenvector of the Simons operator $\tilde{\mathcal{B}}.$ 
\end{enumerate}
\end{thm}

We observe that  elementary Linear Algebra
guarantees the existence of eigenvectors of $\tilde{\mathcal{B}}$ as well as, in the case that $M$ is compact, standard results from Analysis guarantee the existence of a minimizer $\eta$  of the functional $\mathcal{N}$ (if $\mathcal{N}_1(M)\neq\emptyset$ and possibly only in the weak sense). The difficulty to have the harmonicity of the Gauss map $\gamma_{\eta}$ associated to $\eta$
is to guarantee that $\eta$ is    parallel in the normal connection. But existence of parallel unit normal sections in codimension bigger than or equal to $2$ seems to be very restrictive. For example, any minimal surface immersed in any space  form admitting a parallel unit normal section cannot be substantial.  Indeed: If the shape operator of a parallel normal section is zero then it is easy to see that the surface is contained in a totally geodesic hypersurface. If not, it follows from the Ricci equation and the minimality that the surface has flat normal bundle and then parallel one dimensional first normal bundle. This implies that the surface has actually substantial codimension one (see Chapter 2 of \cite{DT}). 

Nevertheless, as explained below,  there are some interesting cases where Theorem \ref {trn} can be applied, as to constant mean curvature hypersurfaces of spheres and to isoparametric submanifolds.

 Although not being able to prove, we believe that   the 
parallelism of the unit normal  section in Theorem \ref{trn} is necessary to have the equivalences among itens \eqref{trn1}, \eqref{trn2} and \eqref{trn3}.

We observe that a normal section $\eta$  in $\mathbb{R}^{n+2}$ of a hypersurface of $M$ of $\mathbb{S}^{n+1}$ is parallel (in the normal connection of $M$) if and only if $\eta=a\nu+b\mu$ where $a, b$ are constants, $\nu$ a unit normal section of $M$ in $\mathbb{S}^{n+1}$ and $\mu$ is the position vector of $M$ in $\mathbb{R}^{n+2}.$

\begin{thm}\label{harm}
Let $M$ be an orientable hypersurface of $\mathbb{S}^{n+1},$  $\eta$ a unit normal section of $M$ in $\mathbb{R}^{n+2}$ parallel in the normal connection and let $\theta\in [0,\pi)$ be the angle between $\eta$ and $\mu.$ We have: 
\begin{enumerate}
\item\label{triv1} If $\theta = 0$ or $\theta=\frac{\pi}{2}$ then the Gauss map $\gamma_\eta$ is harmonic if and only if $M$ is a minimal hypersurface. 
\item\label{ntriv2} If $\theta\neq 0$ and $\theta\neq \frac{\pi}{2}$ then the Gauss map $\gamma_\eta$ is   harmonic  if and only if $M$ is a CMC hypersurface of $\mathbb{S}^{n+1} $ and the second fundamental form of $M$ in $\mathbb{S}^{n+1}$ has constant length and satisfies
$$\norma{S_\nu}^2=nH\left(\cot\theta-\tan\theta\right)+n,$$ 
where $\nu$ is a unit normal vector of $M$ in $\mathbb{S}^{n+1}$ 
 and $H$ is the mean curvature of $M$ with respect to $\nu.$
\end{enumerate}
\end{thm}

Item \eqref{triv1}  of Theorem~\ref{harm} is essentially known. Indeed, the case  $\theta=0$ is included in the Theorem 3 of \cite{Ta} and the case $\theta=\frac{\pi}{2}$  can be proved from the results of  \cite{ChenX} and \cite{I} with the Gauss map as defined by Obata \cite{O}.  Also, although not explicitely stated in \cite{JS}, the case $\theta=\frac{\pi}{2}$  follows from its results too  (Section 5 of \cite{JS}).

In the next result we apply  Theorem \ref{harm} to characterize the minimal Clifford torus and the $H(r)$-torus in the sphere $\mathbb{S}^{n+1}$  in terms of the harmonicity of the Gauss maps associated to unit normal sections in $\mathbb{R}^{n+2}$ parallel in the normal connection. We  make use of  theorems on  CMC hypersurfaces in the sphere  due to  H. Alencar,  S. Chern, M.  Do Carmo and S. Kobayashi \cite{AdC} and  \cite{CdCK}.
We need to recall a notation of \cite{AdC}. 

Let $M$ be a compact and oriented CMC hypersurface of $\mathbb{S}^{n+1}.$ We assume that $M$ is oriented in such way that the mean curvature $H$ of $M$ is non negative. For each $H,$ set 
$$P_H(x)=x^2+\frac{n(n-2)}{\sqrt{n(n-1)}}Hx-n(H^2+1),$$
and let $B_H$ be the square of the positive root of $P_H(x)=0.$

We observe that when $H\neq 0$ the equation on $\theta$
\begin{equation}\label{angle2}
nH\left(\cot\theta-\tan\theta\right)= B_H+nH^2-n
\end{equation}
has exactly two solutions $\theta_1, \theta_2\in (0,\pi),$ where $\theta_1\in (0,\frac{\pi}{2})$ and $\theta_2=\theta_1+\frac{\pi}{2}.$ These solutions correspond to two parallel unit normal sections of $M$ in $\mathbb{R}^{n+2}$ having angles $\theta_1, \theta_2$ with $\mu.$

\begin{corol}\label{corol3}
Let $M$ be a compact  CMC hypersurface of $\mathbb{S}^{n+1}\subset \mathbb{R}^{n+2}.$     Then, 
\begin{enumerate}
\item Case $H=0$:  The following alternatives are equivalent:
\begin{enumerate}
\item\label{corol3ia} For any parallel unit normal section $\eta$ of $M$ in $\mathbb{R}^{n+2}$ the Gauss map $\gamma_\eta$ is a harmonic map.
\item\label{corol3ib} There exists a parallel unit normal section $\eta$ of $M$ in $\mathbb{R}^{n+2}$ where the angle  between $\eta$ and $\mu$ is neither $0$ nor $\frac{\pi}{2}$ and the associated Gauss map $\gamma_\eta$ is a harmonic map. 
\item\label{corol3ic} $M$ is a minimal Clifford torus, that is,
$$M=\mathbb{S}^{k}\left(\left(\frac{k}{n}\right)^\frac{1}{2}\right)\times \mathbb{S}^{n-k}\left(\left(\frac{n-k}{n}\right)^\frac{1}{2}\right). $$
\end{enumerate}
\item\label{corol3i2}  Case $H\neq 0$ and $n> 2:$   

\begin{enumerate}

\item $M$ is an  $H(r)$-torus with $r^2<(n-1)/{n}$
if and only if the only harmonic Gauss maps of $M$ are, up to sign, the two parallel unit normal sections determined by the two solutions of equation \eqref{angle2}.

\item Let $\eta$ be is a parallel unit normal section of $M$.  If  the Gauss map $\gamma_\eta$ is a harmonic map and the angle $\theta$ between $\eta$ and $\mu$ satisfies
\begin{equation*}
nH\left(\cot\theta-\tan\theta\right)< B_H+nH^2-n.
\end{equation*} 
then  $M$ is a totally umbilical hypersurface of $\mathbb{S}^{n+1}$  with principal curvature $\lambda>0$ and
\begin{equation*}
\cot\theta-\tan\theta=\frac{\lambda^2-1}{\lambda}.
\end{equation*}
\end{enumerate}

\item Case $H\neq 0$ and $n=2:$ Let $\eta$ be a parallel unit normal section of $M$ in $\mathbb{R}^{n+2},$ and  let $\theta$ the angle between $\eta$ and $\mu.$ Then $\gamma_\eta$ is harmonic if and only if $M$ is one of the following two surfaces:

\begin{enumerate}
\item a product of circles $\mathbb{S}^1(r)\times \mathbb{S}^1(\sqrt{1-r^2})$ with $0<r<1,$ $r^2\neq 1/2$ and 
\begin{equation*}\label{corol3_4}
\cot\theta-\tan\theta=\frac{1-2r^2}{r\sqrt{1-r^2}}
\end{equation*}

\item  a totally umbilical hypersurface of the sphere with principal curvature $\lambda\neq 0$ and 
\begin{equation*}\label{corol3_5}
\cot\theta-\tan\theta=\frac{\lambda^2-1}{\lambda}.
\end{equation*}
\end{enumerate}
\end{enumerate}
\end{corol}

A question that arises is if the hypersurfaces of $\mathbb{S}^{n+1}$ appearing in Corollary $\ref{corol3}$ are the only hypersurfaces of the sphere  that admits a parallel unit normal section such that its associated Gauss map is a harmonic map. The answer is no, as one may see in Theorem \ref{isorn} below. 

A well known family of submanifolds of the Euclidean space, studied in the classical and modern theory of submanifolds, where we can apply Theorem \ref{trn},  is the family of  isoparametric submanifolds. Recall that a complete submanifold $M$ of $\mathbb{R}^m$ or $\mathbb{S}^m$ is called isoparametric if it has flat normal bundle and its principal curvature in the direction of any parallel normal vector field is constant \cite{T}.  

The classification of the isoparametric submanifolds of the Euclidean space is a long standing problem in Differential Geometry. It can be reduced to the problem of classifying the isoparametric submanifolds in the spheres since  any isoparametric submanifold of  $\mathbb{R}^m$ is the product of an isoparametric submanifold of a sphere with an affine subspace of $\mathbb{R}^m$ \cite{T}. A classification of the compact isoparametric submanifols of codimension bigger than or equal to 2 in spheres have been obtained during the last decades; however, a classification of the isoparametric hypersurfaces was obtained only recently  (see \cite{Qchi}). As the authors know,  a complete classification is still an open problem. In the next result we obtain what we believe to be a strong property of isoparametric submanifolds of the the Euclidean space,  seemingly not known:

 \begin{thm}\label{isorn}
 Let $M$ be  an isoparametric  submanifold of $\mathbb{R}^{m}.$  Then the Simons operator $\tilde{\mathcal{B}}$ has constant non negative eigenvalues  and there is an orthonormal basis $\eta_{1},\cdots,\eta_{r}$ of $\mathcal{N}_1(M)$ of eigenvectors of $\tilde{\mathcal{B}}$ such that the  Gauss maps $\gamma_{\eta_i}:M\to \mathbb{S}^{m-1},$ $i=1, \cdots, r,$ are harmonic maps. 
\end{thm}

In Section 6  we  extend Theorem \ref{trn} to more general ambient spaces, which include symmetric spaces, and  prove an extension of Theorem \ref{isorn} to minimal isoparametric submanifolds of spheres. A key concept related to the existence of parallel normal sections,  that allows to obtain part of our extensions, is that of \textit{polar action} (see section \ref{HUNS}). Our results also extend and generalizes several results of \cite{SMB}, \cite{SML}, \cite{BR},  \cite{EFFR}, \cite{FR}, \cite{M} and \cite{RR}.

The paper is organized as follows. In Section~\ref{P} we give  basic definitions and set the notation used throughout the paper. 

In Section~\ref{rl} we define and prove several facts about the rough Laplacian of a vector field along a submanifold $M$ of a Riemannian manifold $N.$ In particular we obtain an extension, to submanifolds of arbitrary codimension,  of Proposition 1 of \cite{FR}. This proposition, which gives a formula, in case of hypersufaces, for the Laplacian of a  function of the form $\forma{\eta,V},$ where $\eta$ is a unit normal section of $M$ and $V$ is a Killing vector field of $N$, is fundamental for proving Theorem 1 of \cite{RR}. The extension of this proposition obtained here (Corollary \ref{corol2}) is also  fundamental for the paper. 

In  Section~\ref{HUNS}  we obtain the Euler-Lagrange equation of the critical points of an energy functional  and prove the existence of a harmonic unit normal section on any principal orbit of a polar action.  

In Section~\ref{EuclideanSpace} we prove our main results on the Euclidean space stated above. Finally, in Section~\ref{AGM} we extend the results of the Euclidean space to a class of homogeneous space where a Gauss map  is naturally associated to a unit normal section of a submanifold of the space. A special interesting case  is the sphere $\mathbb S^7$ with the octonionic multiplication.

\section{Preliminaries}\label{P}

Let $M$ be $n-$dimensional smooth manifold immersed in a Riemannian manifold $N.$
Denote by $\mathcal{T}(M)$ and $\mathcal{T}(N)$ the space of smooth sections of $TM$ and $TN$ that is, the space of smooth vector fields of $M$ and $N$ respectively. Set $\mathcal T(M,N):=\mathcal N(M) \oplus\mathcal T(M).$ If $W\in  \mathcal T(M,N)$ then we say that $W$ is a \emph {vector field along} $M.$

We denote by  $\n^\perp:=\left (\n\right )^\perp$  the normal connection of $M,$ acting on $TM^\perp,$ where $\perp$ is the orthogonal projection of $TN$ on $TM^\perp.$ The  curvature tensor ${\rm R}$ and normal curvature tensor $\rm {R}^\perp$ are given by  
\begin{eqnarray*}
{\rm R}(X,Y)Z &=& \n_Y \n_X Z-\n_X\n_Y Z + \n_{[X,Y]} Z,\\
{\rm R}^\perp(X,Y)\eta  & = & \np_Y \np_X \eta-\np_X\np_Y \eta+ \np_{[X,Y]} \eta,
\end{eqnarray*}
for  $X, Y, Z$ vector fields of $M$ and $\eta$ a normal vector field of $M$. 
We say that the normal bundle of $M$ is flat if  ${\rm R}^\perp\equiv 0.$

The mean curvature vector of the immersion of $M$ in $N$  is defined by $\overrightarrow{H}=\frac{1}{n}{\rm tr}(\mathcal{B}),$
where $\mathcal{B}$ denotes the second fundamental form of the immersion of $M$ in $N$ and $n$ is the dimension of $M.$

Next we introduce an operator, the normal Ricci operator,  that appears frequently in our results. First consider the  symmetric bilinear form ${\rm Ric}_M$ on $\mathcal T(M,N)$ defined as follows: Given $Z, W\in \mathcal T(M,N)$ then ${\rm Ric}_M(Z,W)$ is the trace of
$$(X,Y) \in \mathcal T(M)\times \mathcal T(M)  \mapsto \forma{{\rm R}(Z,X)W,Y}.$$
We then define the normal Ricci operator %
 $${\rm Ric}_M^\perp: \mathcal{N}(M)\to \mathcal{N}(M)$$
 of $M$ by the relation 
 $${\rm Ric}_M(\eta_1,\eta_2)=\forma{{\rm Ric}_M^\perp(\eta_1),\eta_2},$$ 
 for all $\eta_1, \eta_2\in \mathcal{N}(M).$ Notice that the normal Ricci operator is self-adjoint.  In particular, if $N$ is a connected spaces of constant curvature $c$ then for any unit normal section $\eta$ of $M$ the normal Ricci operator of $\eta$ is given by
\begin{eqnarray}\label{RicciConstantspaces}
 {\rm Ric}_M^\perp(\eta)=cn\ \eta.
\end{eqnarray} 

We conclude this section giving more details for the computation of the Simons Operator $\tilde{\mathcal{B}}$. The vector bundle homomorphism $ \mathcal{B}:\mathcal{N}(M)\to  \mathcal{S}(M)$ is defined by
 $\mathcal B(\eta)=S_\eta,$ where $S_\eta$ is the second fundamental form associated to $\eta.$ The bundle $\mathcal{S}(M)$ is provided with the Hilbert-Schimidt-metric, that is, for $T_1, T_2\in \mathcal{S}(M)$ we have, at each fibre, 
 \begin{eqnarray*}
\forma{T_1,T_2}=\sum_{i=1}^n \forma{T(E_i),T(E_i)}, 
\end{eqnarray*}
where  ${E_1,\cdots, E_n}$ is a local orthonormal basis of $TM.$  Since, the Simons operator is defined by  $\mathcal{B}^\ast  \mathcal{B}$ where $\mathcal{B}^\ast: \mathcal{S}(M)\to \mathcal{N}(M)$  is the fiber-wise adjoint map of $\mathcal B,$ we have 
\begin{eqnarray}
\forma{\tilde{\mathcal{B}}(\eta_1),\eta_2}=\forma{\mathcal{B}^\ast  \mathcal{B}(\eta_1),\eta_2}=\forma{\mathcal{B}(\eta_1), \mathcal{B}(\eta_2)}=\forma{S_{\eta_1},S_{\eta_2}},
\end{eqnarray}
for  any $\eta_1, \eta_2\in \mathcal{N}(M).$

\section{The rough Laplacian of a vector field along a submanifold}\label{rl}
The notion of the rough Laplacian of a vector field of a Riemannian manifold is well know \cite{U}. 
In this section we define and prove several facts about the rough Laplacian of a vector field along a submanifold $M$ of $N$,  to be used  in the next sections.  Some of the results might have independent interest.

The rough Laplacian $\n^2W$ of  vector field $W\in \mathcal T(M,N)$ along $M$ is also a vector field along $M$ defined by
 $$\n^2 W=\sum_{i=1}^n (\n_{E_i}\n_{E_i}W-\n_{\n_{E_i}^\top E_i}W),$$ 
 where  $n=\dim (M)$ and $\{E_1,\cdots,E_n\}$ is a local orthonormal basis of $M$. One may  see that this formula is given as the trace of a bilinear form and hence does not depend of the orthonormal basis $\{E_j\}.$

We are mainly interested in the case of the rough Laplacin acting on Killing fields.
\begin{exa}
Let $M$ be an $n$-dimensional  manifold immersed in the Euclidean space $\mathbb{R}^{m}$ and let $V$ be a Killing vector field of $\mathbb{R}^{m}.$ One may  see that  
\begin{eqnarray*}
\n^2 V = n\n_{\overrightarrow{H}}V,
\end{eqnarray*}
where $\n$ denotes the Riemannian connection of $\mathbb{R}^{m}.$
\end{exa}

\begin{exa}\label{sm}
Let $M$ be an $n$-dimensional  manifold immersed in the sphere $\mathbb{S}^{m}$ and let $V$ be a Killing vector field of $\mathbb{S}^m.$  

We denote by $\n$ and $\overline{\n}$ the Riemannian connections of $\mathbb{S}^m$ and $\mathbb{R}^{m+1},$  respectively. Note that $V$ can be represented as a constant skew-symmetric matrix $A$ of dimension $m+1,$ that is, $V(p)=Ap,$ where $p$ is  regarded  as column vector of $\mathbb{R}^{m+1}.$

Choose $p\in M$ and let  $\{E_1,\ \cdots,\ E_n\}$ be a local orthonormal basis of $M$ geodesic at $p.$ Then, at $p$ 
\begin{eqnarray*}
\n^2 V&=& \sum_i\n_{E_i}\n_{E_i}V\\
&=& \sum_i\left(\on_{E_i}\left(A E_i-\forma{A E_i,p}p \right)-\forma{\on_{E_i}\left(A E_i-\forma{A E_i,p}p \right),p}p\right)\\
&=& \sum_i\left(A \on_{E_i}E_i-\forma{A \on_{E_i}E_i,p}p-\forma{A E_i,p}E_i\right)\\
&=& nA\overrightarrow{H}-nAp-n\left\langle A\overrightarrow{H},p\right\rangle p+ V^{\top_M}\\
& =& n \n_{\overrightarrow{H}} V- n V+ V^{\top_M}.
\end{eqnarray*}
where $(\cdot)^{\top_M}$ denotes the orthogonal projection on $TM.$
\end{exa}

\begin{exa}
Let $M$ be an $n$-dimensional manifold immersed in the hyperbolic space $\mathbb{H}^m$ and let $V$ be a Killing vector field of $\mathbb{H}^m$ then, 
\begin{eqnarray}\label{hypc}
\n^2 V =  n \n_{\overrightarrow{H}} V+ n V- V^{\top_M},
\end{eqnarray}
where $(\cdot)^{\top_M}$ denotes the orthogonal projection on $M$ and $\n$ is the Riemannian connection of $\mathbb{H}^m.$ This formula can be obtained as the previous case of the sphere by using the Lorentz model of the hyperbolic space. 
\end{exa}

   The next lemma gives an expression for the tangent part of the rough Laplacian of a normal section.

\begin{lem}\label{Tangentpart}
Let $M$ be an $n$-dimensional  manifold immersed in a Riemannian manifold $N.$ Then,  for all $\eta\in \mathcal{N}(M)$ and  $X\in \mathcal T(M)$ it holds the equality
\begin{small}
\begin{equation*}
\forma{\n^2 \eta,X}= {\rm Ric}_M(\eta, X)-n\forma{{\rm grad}\forma{\overrightarrow{H},\eta},X}+ n\forma{\overrightarrow{H},\n_X\eta}-2 {\rm tr}\left(S_{\np_{(\cdot)}\eta}(X)\right),
\end{equation*}
\end{small}
where  ${\rm grad}$ denotes the gradient on $M$  and $\overrightarrow{H}$  the mean curvature vector  of $M.$
\end{lem}

\begin{proof}
Choose $p\in M$ and let  $\{E_1,\ \cdots,\ E_n\}$ be a local orthonormal frame of $M$ geodesic at $p.$ Then, 
$$(\n^2 \eta)(p)=\sum_{i=1}^n \n_{E_i}\n_{E_i}\eta.$$ 
The following equalities are satisfied 
 $$\forma{\n_{E_i}[E_i,E_j],\eta}=\forma{[E_i, [E_i,E_j]],\eta}=0$$
  and $\forma{E_j,\eta}=0,$ along  $M.$  Then, for each $j\in\{1,\cdots, n\}$ we have at  $p$ 
\begin{footnotesize}
\begin{eqnarray*}
\forma{\n^2 \eta, E_j} &=& -\sum_{i=1}^n\left( \forma{\n_{E_i}\n_{E_i}E_j,\eta}+2\forma{\n_{E_i}E_j,\n_{E_i}\eta}\right)\\
&=& -\sum_{i=1}^n \left(\forma{\n_{E_i}[E_i,E_j],\eta}+\forma{\n_{E_i}\n_{E_j}E_i,\eta}+2\forma{\n_{E_i}E_j,\n_{E_i}^\perp\eta}\right)\\
&=& -\sum_{i=1}^n \left(\forma{\n_{E_i}\n_{E_j}E_i,\eta}\right)-2{\rm tr}\left(S_{\np_{(\cdot)}\eta}(E_j)\right)\\
&=& \sum_{i=1}^n \left( \forma{R(E_i,E_j)E_i,\eta} -\forma{\n_{E_j}\n_{E_i}E_i,\eta}\right)-2
{\rm tr}\left(S_{\np_{(\cdot)}\eta}(E_j)\right)
\\
&=& \sum_{i=1}^n \left( \forma{R(E_i,E_j)E_i,\eta} -E_j\forma{\n_{E_i}E_i,\eta}+\forma{\n_{E_i}E_i,\n_{E_j}\eta}\right)-\\
& & 2{\rm tr}\left(S_{\np_{(\cdot)}\eta}(E_j)\right)\\
&=&  {\rm Ric}_M(\eta,E_j)-n\forma{{\rm grad}\forma{\h,\eta},E_j}+n\forma{\h,\n_{E_j}\eta}-2{\rm tr}\left(S_{\np_{(\cdot)}\eta}(E_j)\right).
\end{eqnarray*}
\end{footnotesize}
The lemma is then proved by observing that $p$ is arbitrary and writing $X$ on the basis $\{ E_j\}.$ 
 \end{proof}

\begin{corol}\label{paral1}
Let $M$ be an $n$-dimensional  manifold immersed in a Riemannian manifold $N$ and assume that $\eta$ is a parallel normal section in the normal connection of $M.$ Then, 
$$ \forma{\n^2 \eta,X}= {\rm Ric}_M (\eta, X)-n\forma{{\rm grad}\forma{\overrightarrow{H},\eta},X},$$
for all $X\in\mathcal{T}(M).$ In particular,  if $M$ is a two-sided hypersurface of $N$ and $\eta$ is a unit normal section then
 $$ \forma{\n^2 \eta,X}= {\rm Ric}(\eta, X)-n\forma{{\rm grad}H,X},$$
where  $H$ is the mean curvature of $M$ with respect to $\eta$  and ${\rm Ric}(\cdot, \cdot)$ is the Ricci tensor of $N.$
\end{corol}
 
In the next lemma we relate the normal part of the rough Laplacian of a unit normal section $\eta$ parallel in the normal connection with the Simons operator of $\eta.$ 

\begin{lem}\label{n2eta}
Let $M$ be a manifold immersed in a Riemannian manifold $N$ and assume that $\eta$ is a parallel normal section in the normal connection of $M.$ Then, 
\begin{eqnarray}
(\n^2\eta)^\perp=-\tilde{\mathcal{B}}(\eta)
\end{eqnarray}
 \end{lem}
 
 \begin{proof}
Let $\{E_1,\cdots, E_n\}$ be a local orthonormal frame of $M$ and let  $\nu$ be a normal vector to $M.$ Then differentiating 
 $$\forma{\n_{E_i}\eta,\nu}=\forma{\n_{E_i}^{\perp}\eta,\nu}=0$$
  with respect to   $E_i$ and summing up   we have 
\begin{eqnarray*}
0&=&\sum_{i=1}^n\left(\forma{\n_{E_i}\n_{E_i}\eta,\nu}+\forma{\n_{E_i}\eta,\n_{E_i}\nu}\right)\\
&=& \sum_{i=1}^n\left(\forma{\n_{E_i}\n_{E_i}\eta,\nu}+\forma{S_\eta(E_i),S_\nu(E_i)}\right)\\
&=&  \forma{\n^2\eta,\nu}+\forma{S_\eta,S_\nu}\\
&=& \forma{\n^2\eta,\nu}+\forma{\mathcal{B}(\eta),\mathcal{B}(\nu)}\\
&=& \forma{\n^2\eta,\nu}+\forma{\mathcal{B}^*\mathcal{B}(\eta),\nu)}\\
&=& \forma{\n^2\eta,\nu}+\forma{\tilde{\mathcal{B}}(\eta),\nu)}
\end{eqnarray*}
that is, $\forma{\n^2\eta,\nu}=-\forma{\tilde{\mathcal{B}}(\eta),\nu)}$ then 
\begin{eqnarray*}
(\n^2\eta)^\perp=-\tilde{\mathcal{B}}(\eta).
\end{eqnarray*} 
 \end{proof}
 
In the next result we obtain a formula for  the normal part of the rough Laplacian of a Killing vector field.

\begin{lem}\label{lem1}   
Let $M$ be an $n-$dimensional manifold immersed in a Riemannian manifold $N.$ If $V$ is a Killing vector field of $N$ and $\eta$ a normal section of $M$ then it holds the equality
\begin{eqnarray}\label{a}
-\forma{\n^2 V,\eta}={\rm Ric}_M(\eta,V)+n\forma{\overrightarrow{H},\n_{\eta}V}.
\end{eqnarray}
\end{lem}

\begin{proof}
 Choose $p\in M$ and let $\{E_1,\ \cdots,\ E_n\}$ be a local orthonormal frame of $M$ geodesic  at $p.$  We extend  $\{E_i\}$ to an open set of $N$ parallelly along the geodesics orthogonal to $M.$  Since $V$ is a Killing vector field we have at $p$
\begin{eqnarray*} 
\forma{\n^2 V, \eta} &=&\sum_{i=1}^n \forma{\n_{E_i}\n_{E_i}V, \eta}\\
&=& \sum_{i=1}^n \left( E_i\forma{\n_{E_i}V,\eta}-\forma{\n_{E_i}V,\n_{E_i}\eta}\right)\\
&=& -\sum_{i=1}^n\left(\forma{\n_{\eta}V,\n_{E_i}E_i}+\forma{\n_{E_i}\n_{\eta}V, E_i}+\forma{\n_{E_i}\eta,\n_{E_i}V}\right).
\end{eqnarray*} 
For the second term of the last expression we obtain
 \begin{eqnarray}\label{eqL1_1}
\sum_{i=1}^n \forma{\n_{E_i}\n_{\eta}V,E_i}=-n\forma{\h,\n_{\eta}V}-\forma{\n^2 V,\eta}-\sum_{i=1}^n \forma{\n_{E_i}\eta,\n_{E_i}V}. \end{eqnarray}
On the other hand, since  $V$ is a Killing vector field of $N$  we have $\forma{\n_{E_i} V,E_i}=0,$    then
\begin{eqnarray}\label{eq0}
0 =\forma{\n_{\eta}\n_{E_i} V,E_i}(p)+\forma{\n_{E_i} V,\n_{\eta}E_i}(p),
\end{eqnarray}
and noting that  $\n_{\eta}E_i(p)=0$ since  $E_i$ is parallel along $\eta,$ we have
 \begin{eqnarray*}
0 =\forma{\n_{E_i} V,\n_{\eta}E_i}(p),
\end{eqnarray*}
then from equation \eqref{eq0} we obtain that $\forma{\n_{\eta}\n_{E_i} V,E_i}(p)=0.$ Thus,
\begin{eqnarray*}
\forma{\rm{R}(\eta,E_i)V,E_i}(p) &=&\forma{\n_{E_i}\n_{\eta}V,E_i}-\forma{\n_{\eta}\n_{E_i} V,E_i}+\forma{\n_{[\eta,E_i]}V,E_i}\\
&=& \forma{\n_{E_i}\n_{\eta}V,E_i}-\forma{\n_{E_i}V,[\eta,E_i]}\\
&=& \forma{\n_{E_i}\n_{\eta}V,E_i}+\forma{\n_{E_i}V,\n_{E_i}\eta},
\end{eqnarray*}
that is, 
\begin{eqnarray}\label{eqL1_2}
\forma{E_i,\n_{E_i}\n_{\eta}V}=\forma{\rm{R}(\eta,E_i)V,E_i}-\forma{\n_{E_i}V,\n_{E_i}\eta}.
\end{eqnarray}
 Thus, from (\ref{eqL1_1}) and (\ref{eqL1_2}) we obtain (\ref{a}).
 
\end{proof}

\begin{corol}
 Let $M$ be  a minimal submanifold of a Riemannian ma\-ni\-fold $N$, $\eta$ a normal section of $M$  and $V$ a Killing vector field of $N.$ Then it holds the equality
\begin{eqnarray*}
\forma{\n^2 V,\eta}=-{\rm Ric}_M(\eta,V).
\end{eqnarray*}
\end{corol}

\begin{lem}\label{lemma1}
Let $M$ be an $n$-dimensional manifold immersed in a Riemannian manifold $N.$ Let $V$ be a Killing field of $N$ and $\eta$ a normal section of $M.$ Then  one has the following formula for the Laplacian of the function $f=\forma{\eta,V}$: 
\begin{footnotesize}
\begin{equation}\label{lemma}
\begin{array}{rcl}
\Delta_M f &=&\forma{\left(\n^2 \eta\right)^\perp,V}-\forma{{\rm Ric}_M^\perp(\eta),V}-n\forma{{\rm grad}\forma{\overrightarrow{H},\eta},V}+ n\forma{\overrightarrow{H},\n_{V^\top}\eta}-\\ && n\forma{\overrightarrow{H},\n_\eta V}+   2\forma{\n\eta,\n V}-2{\rm tr}\left(S_{\np_{(\cdot)}\eta}(V^\top)\right). \end{array}
\end{equation}
\end{footnotesize}
\end{lem}

\begin{proof}
Given $p\in M,$ let $\{E_1,\cdots,E_n\}$ be a local orthonormal basis of $M$ geodesic at $p.$ Then, at $p,$
$$\Delta_M f=\sum_{i=1}^n E_i(E_i(f)).$$ 
Note that
\begin{eqnarray*}
\Delta_M f&=& \forma{\n^2\eta,V}+2\forma{\n \eta,\n V}+\forma{\eta,\n^2 V}.
\end{eqnarray*}
 From this equation and the Lemma~\ref{lem1}, we have
\begin{small}
\begin{equation}\label{prop1eq1}
\begin{array}{rcl}
\Delta_M f&=&\forma{\n^2\eta,V}+2\forma{\n \eta,\n V}-{\rm Ric}_M(\eta,V)-n\forma{\overrightarrow{H},\n_{\eta}V}.
\end{array}
\end{equation}
\end{small}
Writing  $V=V^\perp+V^\top$  we have that
\begin{eqnarray}\label{prop13}
\forma{\n^2\eta,V}&=&\forma{\n^2\eta,V^\top}+\forma{(\n^2\eta)^\perp,V}
\end{eqnarray} 
and 
\begin{eqnarray}\label{prop14}
{\rm Ric}_M(\eta, V^\top)&=&{\rm Ric}_M (\eta, V)-\forma{{\rm Ric}_M ^\perp(\eta), V}. 
\end{eqnarray}
Recall  from Lemma~\ref{Tangentpart} that  
\begin{footnotesize}
\begin{equation}\label{prop1eq2}
\begin{array}{rcl}
\forma{\n^2\eta,V^\top}&=&{\rm Ric}_M (\eta, V^\top)-n\forma{{\rm grad}\forma{\overrightarrow{H},\eta},V^\top}+ n\forma{\overrightarrow{H},\n_{V^\top}\eta}-\\&& 2{\rm tr}\left(S_{\np_{(\cdot)}\eta}(V^\top)\right). \end{array}
\end{equation}
\end{footnotesize}
Then, from equations \eqref{prop1eq1},  \eqref{prop13}, \eqref{prop14} and \eqref{prop1eq2} we obtain \eqref{lemma}.

\end{proof}

As a consequence of Lemma~\ref{lemma1} we obtain the following useful formula.

\begin{corol}\label{corol2}  Let $M$ be an $n-$dimensional manifold immersed in a Rie\-mannian manifold $N,$  $V$ a Killing vector field of $M$ and  $\eta$ a parallel unit normal section of $M.$ If $f:M\to \mathbb{R}$ is given by
 $$f(q)=\forma{\eta(q),V(q)},\ q\in M,$$
  then 
\begin{footnotesize}
\begin{equation}\label{useful}
-\Delta_M f = n\forma{{\rm grad}\forma{\overrightarrow{H},\eta},V}+n\forma{\overrightarrow{H},\n_\eta V} +\forma{\tilde{\mathcal{B}}(\eta),V}+\forma{{\rm Ric}^\perp_M(\eta),V},
\end{equation} 
\end{footnotesize}
where $\overrightarrow{H}$ is the mean curvature vector of $M.$ 
\end{corol}

\begin{proof} 
Given $p\in M$ let   $\{E_1(p),\ \cdots,\ E_n(p)\}$ be a basis of the tangent space $T_pM$ that diagonalizes the second fundamental form of $M$ determined by $\eta.$ Then, since $\eta$ is parallel in the normal connection of $M$ and $V$ is a Killing vector field of $N$ the following terms in the equation~\eqref{lemma} are zero: 
 $$\forma{\n\eta,\n V}= \sum_{i=1}^n\forma{S_\eta(E_i),\n_{E_i} V}=0,$$
\begin{eqnarray*}
\forma{\overrightarrow{H},\n_{V^\top}\eta}= \forma{\overrightarrow{H},\n_{V^\top}^\perp\eta}=0
\end{eqnarray*}
and 
$${\rm tr}\left(S_{\np_{(\cdot)}\eta}(V^\top)\right)=0$$
thus, from  Lemma~\ref{lemma1} we have
\begin{equation*}\label{temp}
\Delta_M f = \forma{\left(\n^2\eta\right)^\perp,V}-\forma{{\rm Ric}_M^\perp(\eta),V}-n\forma{{\rm grad}\forma{\overrightarrow{H},\eta},V}-n \forma{\overrightarrow{H},\n_\eta V},
\end{equation*}
which together with the Lemma~\ref{n2eta} gives \eqref{useful}.
\end{proof}

\section{Harmonic unit normal sections}\label{HUNS}

In this section we introduce the concept of harmonic unit normal sections of a submanifold  $M$ of a Riemannian manifold $N$  and prove that unit normal sections of principal orbits of polar actions provide a family of examples of harmonic unit normal sections. These examples are interesting on its own and important in our applications.

Let $M$ be a compact manifold  immersed in a Riemannian manifold $N.$ Herein,  we assume that $\mathcal{N}_1(M)$ is nonempty.  We define the energy functional  $\mathcal{N}:\mathcal{N}_1(M)\to \mathbb{R}$  by 
\begin{eqnarray*}
\mathcal{N}(\eta)=\int_M \norma{\n \eta}^2.
\end{eqnarray*}
We say that a unit normal section of $M$ is harmonic if it is a critical point of $\mathcal{N}.$ If $M$ is not compact a unit normal section $\eta$ of $M$ is harmonic if $\eta$ is a critical point of $\mathcal{N}$   on any compact subdomain of $M.$ 

Taking variations of a unit normal section $\eta$  on $\mathcal N_1(M)$ we  obtain:

\begin{prop}\label{husP}
Let $M$ be a compact manifold  immersed in a Riemannian manifold $N.$  A unit normal section $\eta \in \mathcal{N}_1(M)$  is  harmonic  if and only if 
\begin{eqnarray}\label{e0}
(\n^2 \eta)^\perp=-\norma{\n \eta}^2 \eta.
\end{eqnarray}
\end{prop}

\begin{proof}
Let $N\in \mathcal{N}(M)$ be a normal vector field of $M.$ Consider the variation $$\eta_N(t)=\frac{\eta+tN}{\norma{\eta+tN}}$$ and differentiating the energy $\mathcal{N}(\eta_N(t))$ at $t=0$ we have
$$\left.\frac{d}{dt}\right\vert_{{t = 0}}\mathcal{N}(\eta_N(t))=-2\int_M\forma{\norma{\n \eta}^2\eta,N}+2\int_M\forma{\n\eta,\n N}.$$
Then,  $\eta\in \mathcal{N}(M)$ is a critical point of $\mathcal{N}$ if and only if 
\begin{eqnarray}\label{e1}
\int_M\forma{\norma{\n \eta}^2\eta,N}=\int_M\forma{\n\eta,\n N}.
\end{eqnarray}
let $\{E_1,\cdots, E_n\}$ be a geodesic frame in a neighbourhood of a fixed point $p\in M.$ Defining the  tangent vector field $$X=\sum_{i=1}^n\forma{\n_{E_i}\eta,N}E_i$$
we have
\begin{eqnarray*}
{\rm div} X&=&\sum_{i=1}^n E_i\forma{\n_{E_i}\eta,N}\\
&=&\forma{\n^2\eta,N}+\forma{\n\eta,\n N}.
\end{eqnarray*}
Noting that this equality is independent of $p$ and integrating we have  
\begin{eqnarray}\label{e2}
\int_M\forma{\n^2\eta,N}=-\int_M\forma{\n\eta,\n N}.
\end{eqnarray}
Thus, from equations \eqref{e1} and \eqref{e2},  since $N$ is arbitrary in $\mathcal{N}(M)$ it follows \eqref{e0}.
\end{proof}

\begin{prop}\label{eighar} Let $\eta$ be a unit normal section of a manifold $M$ immersed in a Rie\-mannian manifold $N$ such that $\eta$ is parallel in the normal connection of $M$. Then, $\eta$ is a harmonic normal section of $M$ if and only if $\eta$ is an eigenvector of $\tilde{\mathcal{B}}.$
\end{prop}

\begin{proof}
Using the fact that $\eta$ is parallel in the normal connection it is easily to see that an analogous proof of Proposition~\ref{husP} shows that $\eta$ is a harmonic normal section of $M$ if and only if $\eta$ satisfies 
$(\n^2 \eta)^\perp=-\norma{\n \eta}^2 \eta.$
Then the proof is completed  noting that from Lemma  \ref{n2eta} holds the equation $(\n^2 \eta)^\perp=-\tilde{\mathcal{B}}(\eta)$.
\end{proof}

An isometric action of a compact Lie group $\mathbb{G}$ on a complete Riemannian manifold $N$ is called polar if there is a connected complete Riemanian manifold $\Sigma$ isometrically immersed  in $N$ which intersects orthogonally all the orbits of $\mathbb{G}$  (see \cite{AB}, Section 4.1). We prove:

\begin{thm}\label{Polar}
If an action of a compact Lie group $\mathbb{G}$ on a Riemannian manifold $N$ is polar then the Simons operator $\tilde{\mathcal{B}}$ on the principal orbits of $\mathbb{G}$ is diagonalizable by parallel unit normal sections with constant eigenvalues. Moreover, these parallel unit eigenvectors of $\tilde{\mathcal{B}}$ are harmonic normal sections. 
\end{thm}

 \begin{proof}
Let $\mu: \mathbb{G}\times N\to N$ be a polar action. For each $g\in \mathbb{G},$ we denote by $\mu^g:N\to N$ the map defined by $\mu^g(\cdot):=\mu(g,\cdot).$ 
 
Let $x\in N$ be such that $\mathbb{G}(x)$ is a principal orbit of $\mathbb{G}$. Let  $\eta_1(x), \cdots,$ $\eta_r(x)$ be an orthonormal basis of eigenvectors of the self-adjoint opera\-tor $\tilde{\mathcal{B}}$ defined on the normal space $T_x^\perp \mathbb{G}(x).$ For each eigenvector $\eta_i(x),$ we define a normal vector field $\widehat{\eta}_i$ along $\mathbb{G}(x)$ by 
 $$\widehat{\eta}_i(\mu(g,x))=d(\mu^g)_x \eta_i(x),\ \ i=1,\ \cdots, r.$$
  It is easy to see that $\widehat{\eta}_i$ is well defined  and is  parallel in the normal connection (see \cite{AB}, Sections 3.4 and 4.1).

We claim that  $\widehat{\eta}_1, \cdots, \widehat{\eta}_r$ are eigenvectors of $\tilde{\mathcal{B}}$ of $\mathbb{G}(x).$ Indeed, let  $\{v_1,\ \cdots,\ v_r\}$ and $\{w_1,\ \cdots,\ w_r\}$ be orthonormal basis of $T_x\mathbb{G}(x)$ that diagonalizes the second fundamental forms  $S_{\eta_l(x)}$ and $S_{\eta_k(x)}$, respectively, $l, k\in \{1,\ \cdots, r\}.$ Thus, if the principal curvatures are denoted by $\lambda_i$ and $\sigma_i$ we have $S_{\eta_l(x)}(v_i)=\lambda_i v_i$ and $S_{\eta_k(x)}(w_i)=\sigma_i w_i.$  

Now, it is easy to see that $S_{\widehat{\eta}_i(\mu(g,x))}=d\mu^g S_{\eta_i(x)} d\mu^{g^{-1}}$ and that the tangent vector fields to $\mathbb{G}(x),$ $V_i(\mu(g,x))=d\mu^g v_i$ and $W_i(\mu(g,x))=d\mu^g w_i,$ are eigenvectors of $S_{\widehat{\eta}_l}$ and $S_{\widehat{\eta}_k}$ with  eigenvalues  equal to $\lambda_i$ and $\sigma_i.$  Hence, if $(a^i_j)$ is the change basis matrix from $\{w_1,\ \cdots,\ w_r\}$ to $\{v_1,\ \cdots,\ v_r\}$ then  $(a^i_j)$ is  the change basis matrix from $\{W_1,\ \cdots,\ W_r\}$ to $\{V_1,\ \cdots,\ V_r\}$ too. 
 As consequence, we have at $\mu(g,x)$
\begin{eqnarray*}
\forma{S_{\widehat{\eta}_l},S_{\widehat{\eta}_k}}&=&\sum_{i=1}^n\forma{S_{\widehat{\eta}_l}(V_i),S_{\widehat{\eta}_k}(V_i)}\\
&=& \sum_{i=1}^n\sum_{j=1}^r\forma{S_{\widehat{\eta}_l}(V_i),S_{\widehat{\eta}_k}(a^j_iW_j)}\\
&=& \sum_{i=1}^n\sum_{j=1}^r a^j_i\forma{S_{\widehat{\eta}_l}(V_i),S_{\widehat{\eta}_k}(W_j)}\\
%&=& \sum_{i=1}^n\sum_{j=1}^r a^j_i\forma{\lambda_i V_i,\sigma_j W_j)}\\
&=& \sum_{i=1}^n\sum_{j=1}^r a^j_i \lambda_i\sigma_j\forma{v_i,w_j}.
\end{eqnarray*}
Therefore, noting that all the terms  $a_j^i$, $\lambda_i$, $\sigma_j$ and $\forma{v_i,w_j}$ are constant  we have that $\forma{S_{\widehat{\eta}_l},S_{\widehat{\eta}_k}}$ is constant  along  $\mathbb{G}(x).$ That is, 
$$\forma{\tilde{\mathcal{B}}(\widehat{\eta}_l),\widehat{\eta}_k}=\forma{S_{\eta_l(x)},S_{\eta_k(x)}}=\delta_{lk}\forma{S_{\eta_l(x)},S_{\eta_k(x)}}.$$
 Since $l$ and $k$ are arbitrary it follows that $\widehat{\eta}_1,\cdots,\widehat{\eta}_r$ is an eigenbasis of $\tilde{\mathcal{B}}$ along $\mathbb{G}(x)$ with constant eigenvalues equal to $\norma{S_{\eta_l(x)}}^2.$

Finally,  from Proposition \ref{eighar} it follows that the parallel eigenvectors $\widehat{\eta}_1,\cdots,\widehat{\eta}_r$ of $\tilde{\mathcal{B}}$ are harmonic unit normal sections of $\mathbb{G}(x).$
 \end{proof}
 
%%%%%%%%%%%%%%%%%%%%%%%%%%%%%%%%%%%%%%%%%%%%%%%%%%%%%%%%%%%%%%%%%%%%%%%%%%%%%%%%%%%%%%%%%%%%%%%%%%%%%%%%%%%%%%%%%%%%%%%%%%%%%%%%%%%%%%%%%%%%%%%%%%%%%%%%%%%%%%%%%%%%%%%%%%%%%%%%%%%%%%%%%%%%%%%%%%%%%%%%%%%%%%%%%%%%%%%%%%%%%%%%%%%%%%%%%%%%%%%%%%%%%%%%%%%%%%%%%%%%%%%%%%%%%%%%%%%%%%%%%%%%%%%%%%%%%%%%%%%%%%%%%%%%%%%%%%%%%%%%%%%%%%%%%%%%%%%%%%%%%%%%%%%%%%%%%%%%%%%%%%%%%%%%%%%%%%%%%%%%%%%%%%%%%%%%%%%%%%%%%%%%%%%%%%%%%%%%%%%%

\section{Harmonicity of the Gauss maps of submanifolds of the Euclidean space}\label{EuclideanSpace}

In this section we prove the results about the harmonicity of the Gauss maps of submanifolds of the Euclidean space. We begin by stating some facts which are direct consequences of our previous results.  

\begin{prop}\label{generlap}
Let $M$ be an $n-$dimensional manifold immersed on the Euclidean space $\mathbb{R}^{m}$  and  let $\eta$ be a unit normal section parallel in the normal bundle of $M$. The Laplacian of the Gauss map $\gamma_\eta$ associated to $\eta$ is given by
\begin{eqnarray*}
 -\Delta\gamma_\eta=n\hspace{0,1cm} {\rm grad}\forma{\overrightarrow{H},\eta}+\tilde{\mathcal{B}}(\eta).
\end{eqnarray*}   
\end{prop}

\begin{proof}
The  Ricci normal operator ${\rm Ric}_M^\perp$ is equal to zero since the curvature tensor of  $\mathbb{R}^m$ vanishes identically. Then, denoting by $\{e_1,\ \cdots,\ e_{m}\}$  the canonical basis of $\mathbb{R}^{m}$ we have, from Corollary~\ref{corol2} 
\begin{equation*}
\Delta \gamma_\eta = \sum_{i=1}^m \Delta_M \forma{\eta, e_i}e_i=-n\hspace{0,1cm}{\rm grad}\forma{\overrightarrow{H},\eta}- \tilde{\mathcal{B}}(\eta).
\end{equation*}
\end{proof}

Recall that a submanifold has parallel normalized mean curvature vector if $\overrightarrow{H}$ is nonzero and the unit normal section $\eta=\overrightarrow{H}/\norma{\overrightarrow{H}}$ is parallel in the normal connection (see \cite{Chen}). Then, from the Proposition~\ref{generlap} we have that an immersed submanifold in the Euclidean space with parallel normalized mean curvature has parallel mean curvature vector if and only if the unit vector $\eta$ in the direction of the mean curvature vector satisfies:  
\begin{eqnarray*}
-\Delta\gamma_\eta=\tilde{\mathcal{B}}(\eta).
\end{eqnarray*}

As a consequence of the Proposition~\ref{generlap} we can prove the equivalences relating  $\eta$ and the associated Gauss map $\gamma_\eta$ of the Theorem~\ref{trn}.

\begin{proof}[Proof of Theorem~\ref{trn}]
From Proposition~\ref{generlap} and since $M$ has parallel mean cuvature vector  we have 
\begin{equation*}
-\Delta \gamma_\eta =\tilde{\mathcal{B}}(\eta).
\end{equation*}
From this equality and recalling that the map $\gamma_\eta$ is harmonic if and only if   $\Delta(\gamma_{\eta})=-f \gamma_{\eta},$ for some function $f:M\to \mathbb{R},$ we obtain that $\gamma_\eta$ is harmonic if and only if $\eta$ is and eigenvector of $\tilde{\mathcal{B}}.$ This proves the equivalence between  \eqref{trn1} and \eqref{trn3}.

The equivalence between \eqref{trn2} and \eqref{trn3} is given in Proposition \ref{eighar}. 
\end{proof}

Let $M$ be an orientable hypersurface of $\mathbb{S}^{n+1}.$ Recall that we denote by $\nu$ a unit normal section of $M$ in $\mathbb{S}^{n+1}$ and by $\mu$  the position vector of $M$ in $\mathbb{R}^{n+2}.$

\begin{prop}\label{lemmasphere}
Let $M$ be an orientable hypersurface of $\mathbb{S}^{n+1},$  $\eta$ a unit normal section of $M$ in $\mathbb{R}^{n+2}$ parallel in the normal connection and let $\theta$ be the angle between $\eta$ and $\mu.$    Then, the Gauss map $\gamma_\eta:M\to \mathbb{S}^{n+1}$ satisfies
$$-\Delta \gamma_\eta=na\hspace{0,1cm} {\rm grad}(H)+(a \norma{S_\nu}^2-nbH)\gamma_\nu+(nb -naH)\gamma_\mu,$$
where $H$ denotes the mean curvature of $M$ in $\mathbb{S}^{n+1}$   with respect to $\nu$, $a=\sin\theta$ and $b=\cos\theta.$
\end{prop}

\begin{proof}
The normal vector $\eta$ parallel in the normal connection is of the form $\eta=a\nu+b\mu,$ with $a, b \in \mathbb{R}$ where $a=\sin\theta$ and $b=\cos\theta.$ Then by Proposition~\ref{generlap} we have 
\begin{eqnarray*}
-\Delta \gamma_\eta &=& %n grad\forma{\overrightarrow{H},\upsilon}+\gamma_{\tilde{\mathcal{B}}(\upsilon)}\\
 na\hspace{0,1cm}{\rm grad}(H) +\tilde{\mathcal{B}}(\eta)
\end{eqnarray*}
and, noting  that, 
\begin{eqnarray*}
\tilde{\mathcal{B}}(\eta)&=& a\tilde{\mathcal{B}}(\nu)+b\tilde{\mathcal{B}}(\mu)\\
%&=& \left(\lambda_1 \norma{S_\eta}^2+\lambda_2\forma{S_\nu,S_\eta}\right)\eta+\left(\lambda_1\forma{S_\eta,S_\nu}+\lambda_2\forma{S_\nu,S_\nu}\right)\\
&=&\left(a \norma{S_\nu}^2-nb H\right)\nu+\left(nb -na H\right)\mu,
\end{eqnarray*}
we conclude the proof.
\end{proof}

We are now in position to prove Theorem~\ref{harm}.

\begin{proof}[Proof of  Theorem~\ref{harm}]
\begin{enumerate}
\item Assume $\theta=0.$ We then have $a=0$ and $b=1$ in  Proposition \ref{lemmasphere}. Also from this proposition   $$-\Delta\gamma_\eta=-nH\gamma_\nu+n\gamma_\mu.$$ Therefore $\gamma_\eta$ is a harmonic map if and only if $H=0.$ 
 Analogously, if $\theta=\frac{\pi}{2},$ then $a=1$ and $b=0$ in  in Proposition \ref{lemmasphere} and $$-\Delta\gamma_\eta=n\hspace{0,1cm} {\rm grad}(H)+\norma{S_\nu}^2\gamma_\nu-nH\gamma_\mu.$$ Hence $\gamma_\eta$ is harmonic if and only if $H=0.$ 

\item Let $\eta^\perp=-b\nu+a\mu$ be a unit normal section of $M$ orthogonal to $\eta,$ where $a=\sin\theta$ and $b=\cos\theta.$ From Proposition~\ref{lemmasphere} it follows that $\gamma_\eta$ is a harmonic map if and only if $H$ is constant and $\forma{-\Delta\gamma_\eta,\eta^\perp}=0.$ Also from this  proposition  we see that  the condition $\forma{-\Delta\gamma_\eta,\eta^\perp}=0$ is equivalent to 
 $$\norma{S_\nu}^2=nH\left(\frac{b}{a}-\frac{a}{b}\right) +n.$$ 
This concludes the proof of the theorem.
\end{enumerate}
\end{proof}

Following  Eells-Lemaire \cite{EL2}, an eigenmap of $M$  is a map $f:M\to \mathbb{S}^k$ such that  $\Delta f = \lambda f$, $\lambda$ constant. The eigenmaps are studied in different contexts, mainly between spheres. The eigenmaps are clearly harmonic, their coordinate functions are eigenfunctions of the Laplacian with the same eigenvalue.

An immediate consequence of   Theorem~\ref{harm} is:

\begin{corol}
Let $M$ be an oriented hypersurface of $\mathbb{S}^{n+1}\subset \mathbb{R}^{n+2}$ and $\eta$ a parallel unit normal section of $M$ with $\theta\neq\frac{\pi}{2}.$   The Gauss map $\gamma_\eta$ is   harmonic  if and only if  $\gamma_\eta$ is an eigenmap.
\end{corol}

\begin{proof}[Proof of  Corollary~\ref{corol3}]
First we define a linear map $\phi$ on the tangent bundle of $M$ by $$\forma{\phi(X),Y}=H\forma{X,Y}-\forma{S_\nu(X),Y},$$ and note that \begin{eqnarray}\label{phi}
\norma{\phi}^2=\norma{S_\nu}^2-nH^2.\end{eqnarray}
\begin{enumerate}
\item Case $H=0:$ \eqref{corol3ia} $\Rightarrow$ \eqref{corol3ib} it is trivial. 

\eqref{corol3ib} $\Rightarrow$ \eqref{corol3ic} Let $\eta$ be a unit normal section of $M$ such that the angle $\theta$ between $\eta$ and $\mu$ is neither $0$ nor $\frac{\pi}{2}$ and $\gamma_\eta$ is a harmonic map.  By item \eqref{triv1} of Theorem~\ref{harm} we have that $\norma{S_\nu}^2=n.$ Then, $M$ is a minimal Clifford torus of $\mathbb{S}^{n+1}$ since  they  are the only compact minimal hypersurfaces of $\mathbb{S}^{n+1}$  such that quality $\norma{S_\nu}^2=n$ is satisfied, see \cite{CdCK}.  

\eqref{corol3ic} $\Rightarrow$ \eqref{corol3ia} Let $M$ be a minimal Clifford torus and let $\eta$  a unit normal section of $M$ in $\mathbb{R}^{n+2}.$ If the angle $\theta$  between $\eta$ and $\mu$ is $0$ or $\frac{\pi}{2}$ then from Theorem \ref{harm}  $\gamma_\eta$ is a harmonic map since $H=0$. If $\theta\neq 0$ and $\theta\neq \frac{\pi}{2}$ then from  Theorem~\ref{harm} $\gamma_\eta$ is a harmonic map since  $\norma{S_\nu}^2=n.$ 

\item Case $H\neq 0$ and $n>2:$
\begin{enumerate}
\item  On one hand, from Theorem 1.5 of \cite{AdC}, the $H(r)$-torus  with $r^2< (n-1)/n$ and $n>2$ is the only compact CMC hypersurface of the sphere $\mathbb{S}^{n+1}$ with $\norma{S_\nu}^2=B_H+nH^2.$  On the other hand, from Theorem~\ref{harm} above, $\gamma_\eta$ is a harmonic map if and only if 
$\norma{S_\nu}^2=nH(\cot\theta-\tan\theta)+n.$ Therefore, 
$M$ is the  $H(r)$-torus with $r^2<(n-1)/{n}$
if and only if the only harmonic Gauss maps of $M$ are, up to sign, the two ones determined by the solutions of equation $nH(\cot\theta-\tan\theta)=B_H+nH^2-n.$

\item  From Theorem~\ref{harm} we have that  \begin{eqnarray}\norma{S_\nu}^2=nH(\cot\theta-\tan\theta)+n \label{ccc}\end{eqnarray} since $\gamma_\eta$ is harmonic. Then, the inequality \eqref{corol3_5} is equivalent to 
$\norma{\phi}^2<B_H.$ From Theorem 1.5 of \cite{AdC} $M$ is a totally umbilical hypersurface of $\mathbb{S}^{n+1}.$ Since, $H\neq 0$ it follows from (\ref{ccc}) that the principal curvature $\lambda$ of $M$ is nonzero and satisfies $$\cot\theta-\tan\theta=\frac{\lambda^2-1}{\lambda}.$$
\end{enumerate}

\item Case $n=2$ and $H\neq 0:$  From Theorem~\ref{harm} it follows that  $\gamma_\eta$ is harmonic  if and only if the mean curvature $H$ and the square  of norm of the second fundamental form $S_\nu$ are constant. Hence, since $n=2,$  the principal curvatures of $M$ in $\mathbb{S}^3$ are constant. It follows that $M$ is a product of circles $\mathbb{S}^1(r)\times \mathbb{S}^1(\sqrt{1-r^2})$ with $0<r<1$ $r^2\neq 1/2$  or $M$ is a totally umbilical surface with $\lambda\neq 0$ (see  Theorem 3.29 of  \cite{Cecil}).  Noting that $\norma{S_\nu}=nH(\cot\theta-\tan\theta)+n$ we have 
\begin{equation*}
\cot\theta-\tan\theta=\frac{1-2r^2}{r\sqrt{1-r^2}}.
\end{equation*}
 when $M$ is the product of circles since the principal curvatures are $\frac{\sqrt{1-r^2}}{r}$ and $-\frac{r}{\sqrt{1-r^2}}$
and 
$$\cot\theta-\tan\theta=\frac{\lambda^2-1}{\lambda}$$ when $M$ is a totally umbilical surface.
\end{enumerate}
\end{proof}

\begin{proof}[Proof of Theorem~\ref{Classification}]
Assume that $M$ admits a unit normal section $\eta$ in $\mathbb{R}^4$ such that $\gamma_\eta$ is a harmonic map.

 We claim that  $M$ is a not minimal surface of $\mathbb{R}^4.$ Indeed: Let $\nu$ be a local unit normal section of $M$  orthogonal to $\eta$ and let $\{e_1, e_2, e_3, e_4\}$  be a parallel orthonormal basis of $\mathbb{R}^4$ such that $e_1$ and $e_2$ are eigenvectors of $S_\nu$ at a point $p$ of $M.$ Then, from Lemma~\ref{lemma1}
\begin{eqnarray*}
\Delta\gamma_\eta(p)&=&\sum_{i=1}^4 \forma{\Delta\gamma_\eta(p),e_i}e_i\\
&=& \sum_{i=1}^4\left(\forma{(\n^2\eta)^\perp(p),e_i}-2{\rm tr}\left(S_{\np_{(\cdot)}\eta}({e_i^\top})\right)\right)e_i\\
&=& \sum_{i=3}^4\forma{(\n^2\eta)^\perp(p),e_i}e_i-2\sum_{i=1}^2{\rm tr}\left(S_{\np_{(\cdot)}\eta}({e_i})\right)e_i.\\
\end{eqnarray*}

Since $\gamma_\eta$ is harmonic  the tangent part of the Laplacian of $\gamma_\eta$ is zero, thus at $p$ we have, for $i,j=1,2$ with $i\neq j,$
\begin{eqnarray*}
0&=&{\rm tr}\left(S_{\np_{(\cdot)}\eta}({e_i})\right)\\
&=&\forma{S_{\np_{e_i}\eta}({e_i}),e_i}+\forma{S_{\np_{e_j}\eta}({e_i}),e_j}\\
&=&\forma{\n^\perp_{e_i}\eta,\nu}\forma{S_{\nu}({e_i}),e_i}+\forma{\n^\perp_{e_j}\eta,\nu}\forma{S_{\nu}({e_i}),e_j}\\
&=& \forma{\n^\perp_{e_i}\eta,\nu}\forma{S_\nu(e_i),e_i}.
\end{eqnarray*}

Note that  $\forma{S_\nu(e_i),e_i}\neq 0$ since the trace of $S_\nu$ is zero and the second fundamental form of $M$ spans the normal space of $S.$ Therefore, $\forma{(\n^\perp_{e_j}\eta)(p),\nu(p)}=0$ for $j=1,2$, that is, $\eta$ is parallel in the normal connection. 

Since the codimension of $M$ is $2$  its the normal bundle is flat. This contradicts the fact that the second fundamental form of $M$ spans the normal bundle of $M$ since a minimal surface of $\mathbb{R}^m$ with flat normal bundle is contained in a $\mathbb{R}^3$ (see Chapter 2 of \cite{DT}). Then, $M$ is not a minimal surface of $\mathbb{R}^4$ with substantial codimension $2$. 

Therefore, from the classification of surfaces with parallel mean curvature vector of $\mathbb{R}^4$ given by  B. Chen \cite{Bchen} and S. Yau \cite{Yau} we have, up to isometries of $\mathbb{R}^4,$ that  $M$ is  a CMC surface of $\mathbb{S}^3\subset\mathbb{R}^4.$ As a consequence $\eta$ can be written (locally) in the form 
 $$\eta=a\nu+b\mu,$$
where $\nu$ is a local unit normal section of $M$ in $\mathbb{S}^3,$ $\mu$ the position vector of $M$ and $a, b$ are differentiable functions defined in $M$ such that $a^2+b^2=1.$
 
 We  now consider the case \eqref{abno} of the theorem that is,  $a$ and $b$ are not constants. Let $\{e_1,e_2,e_3,e_4\}$ be a parallel orthonormal basis  of  $\mathbb{R}^4$ such that $e_1$ and $e_2$ are eigenvectors of $S_{\nu}$ at a point $p\in M.$ 

We claim that: for $i=1, 2$ that it holds
\begin{eqnarray}\label{reduction}
\forma{{\rm grad}\forma{\overrightarrow{H},\eta},e_i}=\forma{\overrightarrow{H},\n_{e_i}\eta},
\end{eqnarray} where $H$  is the mean curvature of $M$ in $\mathbb{S}^3$ with respect to $\nu.$ Indeed,
\begin{eqnarray*}
\forma{{\rm grad}\forma{\overrightarrow{H},\eta},e_i}&=&\forma{{\rm grad}\forma{H\nu+\mu,f_1\nu+f_2\mu},e_i}\\
&=& \forma{{\rm grad}(Hf_1+f_2),e_i}\\
&=& e_i(f_1)H+e_i(f_2)\\
&=& \forma{H\nu+\mu,e_i(f_1)\nu+e_i(f_2)\mu}\\
&=& \forma{\overrightarrow{H},\n_{e_i}\eta}.
\end{eqnarray*}

Then from the from equation \eqref{reduction} and  Lemma~\ref{lemma1}  we have for $ i=1, 2$ 
\begin{eqnarray}\label{lapla0}
\Delta\forma{\eta,e_i}(p)&=&-2{\rm tr}\left(S_{\n^\perp_{(\cdot)}\eta}(e_i)\right).
\end{eqnarray}

From the harmonicity of $\gamma_\eta$ we have that  the Laplacian  $\Delta\gamma_\eta(p)$ has not tangent part. Since the Laplacian of $\gamma_\eta$ at $p$ is 
$$\Delta\gamma_\eta(p)=\sum_{i=1}^4\Delta\forma{\eta,e_i}(p)e_i,$$ we have
$${\rm tr}\left(S_{\n^\perp_{(\cdot)}\eta}(e_i)\right)=0, \text{ for}\ i=1, 2.$$

The last equation implies a condition on the second fundamental  determined by  $\eta^\perp=-f_2\nu+f_1\mu.$ Indeed: first note that $\eta^\perp$ is a unit normal section of $M$ orthogonal to $\eta.$  We then have, for $i, j=1, 2$ with $i\neq j$
\begin{eqnarray*}
0&=&{\rm tr}\left(S_{\n^\perp_{(\cdot)}\eta}(e_i)\right)\\
&=& \forma{S_{\n^\perp_{e_i}\eta}(e_i),e_i}+\forma{S_{\n^\perp_{e_j}\eta}(e_i),e_j}\\
&=&  \forma{\n^\perp_{e_i}\eta,\eta^\perp}\forma{S_{\eta^\perp}(e_i),e_i}+\forma{\n^\perp_{e_j}\eta,\eta^\perp}\forma{S_{\eta^\perp}(e_i),e_j}\\
&=&\forma{\n^\perp_{e_i}\eta,\eta^\perp}\forma{S_{\eta^\perp}(e_i),e_i},
\end{eqnarray*}
as claimed. 

We now observe that  none of the following two equalities can occur: 
\begin{eqnarray*}
\forma{\n^\perp_{e_1}\eta,\eta^\perp}=0  &\text{and}&  \forma{\n^\perp_{e_2}\eta,\eta^\perp}=0,\\
\forma{S_{\eta^\perp}(e_1),e_1}=0 &\text{and}& \forma{S_{\eta^\perp}(e_2),e_2} =0.
\end{eqnarray*}

Indeed, the first equality implies that $\eta$ is parallel in the normal connection and then  $a$ and $b$ are constants,   a contradiction!  The second equality  contradicts  the hypothesis that the second fundamental form of $M$ spans the normal space of $M$. 

Therefore,  we have 
\begin{eqnarray*}
\forma{\n^\perp_{e_i}\eta,\eta^\perp}=0 &\text{and}& \forma{S_{\eta^\perp}(e_j),e_j}=0, 
\end{eqnarray*}
for $i, j=1, 2$  and $i\neq j.$

From $\forma{S_{\eta^\perp}(e_j),e_j}=0$  we obtain  $a=\lambda_j b,$ where $\lambda_j$ is the principal curvature of $M$ with respect to $e_j.$ It follows that $$a=\pm \frac{\lambda_j}{\sqrt{1+\lambda^2_j}}\ \  \text{and}\ \  b=\pm \frac{1}{\sqrt{1+\lambda^2_j}}.$$ Moreover,   for $i\neq j$ one has $\forma{\n^\perp_{e_i}\eta,\eta^\perp}=0 $ and hence  $e_i(\lambda_j)=0$ that is, $\lambda_j$ is a non constant principal curvature which it is constant along  the principal direction $e_i.$  We  have then  proved item \eqref{abno} of the  corollary.

It remains to consider the case when $a$ and $b$ are constants. 
If $a=0$ or $b=0$ then, from  Theorem~\ref{harm},    $M$ is a minimal surface of $\mathbb{S}^3.$ It is the alternative \eqref{class1c}.

Assume that $a, b$ are non zero constants. Then, from Theorem~\ref{harm},  the second fundamental form of $M$ has constant length. Since $M$ is a CMC surface of $\mathbb{S}^3$ it  follows that the principal curvatures of $M$ are constant.  Then,  $M$ is contained in a totally umbilical non totally geodesic surface of $\mathbb{S}^3$ or $M$ is contained in an isoparametric surface of $\mathbb{S}^3$ with two principal curvatures, that is, $M$ is an open subset  of the product of circles $\mathbb{S}^1(r)\times \mathbb{S}^1(\sqrt{1-r ^2})\subset \mathbb{S}^3\subset \mathbb{R}^4$ with $0<r<1$ (see  Theorem 3.29 of  \cite{Cecil} ). 

We now prove the converse of items  \eqref{class1a}, \eqref{class1b} and \eqref{class1c}. If $M$ is a minimal surface and $\eta$ is the position vector or $M$ or $\eta=\nu$ then from Theorem \ref{harm}  $\gamma_\eta$ is a harmonic map.  This proves the converse of item \eqref{class1c}. 

If $M$ is an open subset of an umbilical surface of $\mathbb{S}^3$ or $M$ is an open subset of $\mathbb{S}^1(r)\times \mathbb{S}^1(\sqrt{1-r^2}).$  Then, from Theorem \ref{harm} the harmonic Gauss maps are given by the solutions for $\theta$ of 
$$\norma{S_\nu}=nH(\cot\theta-\tan\theta)+n,$$ where $H$ is the mean curvature of $M$ with respect to $\nu.$  This proves the converse of items \eqref{class1a} and \eqref{class1b}. 
 \end{proof}

\begin{proof}[Proof of Theorem~\ref{nhS4}]
The proof is analogous to the first part of proof of Theorem \ref{Classification}. Assume by contradiction that the minimal surface $M$ of $\mathbb{S}^4\subset\mathbb{R}^5$ admits a  unit normal section $\eta$ of $M$ such that $\gamma_\eta$ is a harmonic map.

 Let $\nu$ be a local unit normal section of $M$ in $\mathbb{S}^4$ such that $\forma{\nu,\eta}=0.$ 
 We consider a parallel orthonormal basis $\{e_1, e_2, e_3, e_4, e_5\}$  of $\mathbb{R}^5$ such that $e_1$ and $e_2$ are eigenvectors of $S_\nu$ at $p.$   Then by Lemma \ref{lemma1}  
\ we have
\begin{eqnarray*}
\Delta\gamma_\eta(p) &=& \sum_{i=1}^5 \forma{\Delta\gamma_\eta(p),e_i}e_i\\
&=& \sum_{i=1}^5\left(\forma{(\n^2\eta)^\perp(p),e_i}-2\forma{\mu,\n_{e_i^\top}\eta}-2{\rm tr}\left(S_{\np_{(\cdot)}\eta}({e_i^\top})\right)\right)e_i\\
&=& \sum_{i=3}^5\forma{(\n^2\eta)^\perp(p),e_i}e_i-2\sum_{i=1}^2{\rm tr}\left(S_{\np_{(\cdot)}\eta}({e_i})\right)e_i\\
\end{eqnarray*}

From the harmonicity of $\gamma_\eta$ we have at  $p,$ for $i, j=1, 2$  with $i\neq j$ that 
\begin{small}
\begin{eqnarray*}
0&=&{\rm tr}\left(S_{\np_{(\cdot)}\eta}({e_i})\right)\\
&=&\forma{S_{\np_{e_i}\eta}({e_i}),e_i}+\forma{S_{\np_{e_j}\eta}({e_i}),e_j}\\
&=&\forma{\n^\perp_{e_i}\eta,\nu}\forma{S_{\nu}({e_i}),e_i}+\forma{\n^\perp_{e_i}\eta,\mu}\forma{S_{\mu}({e_i}),e_i}+\forma{\n^\perp_{e_j}\eta,\nu}\forma{S_{\nu}({e_i}),e_j}
+ \\&& \forma{\n^\perp_{e_j}\eta,\mu}\forma{S_{\mu}({e_i}),e_j}\\
&=&\forma{\n^\perp_{e_i}\eta,\nu}\forma{S_\nu(e_i),e_i}.
\end{eqnarray*}
\end{small}
If the second fundamental form of $M$ spans the normal space at each point then $\eta$ is parallel in the normal connection.  The minimality of $M$ then imples that its substantial codimension can be reduced, contradiction! This concludes the proof of the theorem.
\end{proof}

We state the next result we recall  the classical notion of isoparametric submanifold see (\cite{T}): 
\begin{defin}\label{isopara}
A complete submanifold $M$ of $N$ ($N = \mathbb{R}^m$, $\mathbb{S}^m$ or $\mathbb{H}^m$) is called isoparametric if the normal bundle of $M$ is flat and the principal curvatures of $M$ in the directions of any parallel normal vector field are constant.
\end{defin}

\begin{proof}[Proof of  Theorem~\ref{isorn}]
Choose  $p\in M$ and  let $\{\eta_1(p),\ \cdots,\ \eta_r(p)\}$ be an orthonormal basis of $T_p M$ of eigenvectors of $\tilde{\mathcal{B}}.$ Since the normal bundle of $M$ is globally flat (see \cite{T})) we can extend $\{\eta_1(p),\ \cdots,\ \eta_r(p)\}$ to an orthonormal basis $\{\eta_1,\cdots, \eta_r\}$ of the normal bundle of $M$ such that each vector field of this basis is parallel in the normal bundle. Moreover, the principal curvatures of the shape operators $S_{\eta_l},$ $l=1,\cdots, r,$ are constant. Then, each function  $\forma{S_{\eta_l},S_{\eta_k}}$ is constant and hence
$$\forma{\tilde{\mathcal{B}}(\eta_i),\eta_j}=\forma{\tilde{\mathcal{B}}(\eta_i),\eta_j}_p=\delta_{ij}.$$
 Therefore, $\eta_1,\ \cdots, \eta_r$ are eigenvectors of $\tilde{\mathcal{B}}.$ By Theorem~\ref{trn} each $\gamma_{\eta_k}$ is a harmonic map of $M$.
 
\end{proof}

The influence of the  image of the Gauss map of a minimal or  a CMC hypersurface $M$ of $\mathbb R^m$ on the geometry and topology of $M$ is a classical topic of study in Differential Geometry. We obtain here:

\begin{thm}
Let $M$ be a compact  $n$-dimensional manifold immersed in $\mathbb{R}^{n+r}$ with parallel mean curvature vector. Let $\eta_1,\cdots, \eta_k$ be  parallel unit normal vector fields in the normal connection that are linearly independent. If $\eta_1,\cdots, \eta_k$  are eigenvectors of $\tilde{\mathcal{B}}$  and the images of the associated Gauss maps $\gamma_{\eta_1},\ \cdots,\ \gamma_{\eta_k}$   lie each in an open  hemisphere of the sphere $\mathbb{S}^{n+r-1}$ then $k<r$ and the codimension of $M$ can be reduced to $r-k$ that is, $M$ is contained in a totally geodesic submanifold of dimension $n+r-k$ of $\mathbb R^{m+r}$.
\end{thm}

\begin{proof}
 From Theorem~\ref{trn} it follows that the Gauss maps  $\gamma_{\eta_1},\ \cdots,\ \gamma_{\eta_k}:M\to \mathbb{S}^{n+r-1}$ are harmonic and satisfy
\begin{equation}\label{laplaeuch} 
\Delta\gamma_{\eta_i}=-\norma{S_{\eta_i}}^2\gamma_{\eta_i},
\end{equation}
$i=1,\cdots, n.$ Equation~\eqref{laplaeuch} and the hypothesis on the images of the Gauss maps imply  that for some vectors $v_1,\ \cdots,\ v_k,$  the  functions $\forma{\eta_1,v_1},\ \cdots,\ \forma{\eta_k,v_k}$ are positive and superharmonic. Thus, each $\forma{\eta_i,v_i}$ is constant and non vanishing. Then, from \eqref{laplaeuch}, the shape operators $S_{\eta_i}$ vanishes identically, $i=1,\cdots, k$. 
Consequently, the parallel subbundle $L=\left(Span\{\eta_1,\ \cdots,\ \eta_k\}\right)^\perp$ contains, for all $p\in M,$ the first normal space of the immersion 
$$N_1(p):=Span\{B(X,Y)\mid X, Y \in T_p M\}.$$ 

 Therefore, from Proposition 2.1 of \cite{DT}, we conclude that the codimension of $M$ can be reduced to  the  rank of the subbundle $L.$ Clearly $k<r$ otherwise $M$ should be totally geodesic which is not possible since $M$ is compact.
\end{proof}

\section{Harmonicity of the Gauss maps of submanifolds of homogeneous spaces}\label{AGM}

Gauss  maps of oriented hypersurfaces have been defined and studied in several ambient spaces (for example  \cite{SMB}, \cite{SML}, \cite{BR}, \cite{BN}, \cite{BD}, \cite{DFM}, \cite{EFFR}, \cite{FM}, \cite{HOS}, \cite{LR}, \cite{M},
 \cite{M} and \cite{RR}), and different versions of the Ruh-Vilms theorem were obtained.

The main goal of this  section is  extend to symmetric spaces the results of the previous section and several results about the harmonicity of the Gauss maps of some of the references cited above. In the Sections~\ref{hs} and \ref{s7} we  recall and adapt to our context the constructions of \cite{SML}, \cite {BR} and \cite{RR}.

\subsection{Homogeneous spaces}\label{hs}

Let $N$ be a homogeneous space and let  $\mathbb G$ be the connected component containing the identity of the the  isometry group of $N.$ Then $N$ is isometric to 
 $$\mathbb{G}/\mathbb{K}=\{x\mathbb{K}\mid x\in \mathbb{G}\}$$ 
where $\mathbb K$ is the isotropy subgroup of $\mathbb G$ at some point of $N$ and the metric of $\mathbb{G}/\mathbb{K}$ comes from a metric in $\mathbb G$ such that the projection $\pi:\mathbb{G}\to \mathbb{G}/\mathbb{K}$ is a pseudo Riemannian submersion. We assume that the metric in  $\mathbb{G}$ is a pseudo bi-invariant metric. 

The set of horizontal vectors of the projection $\pi:\mathbb{G}\to \mathbb{G}/\mathbb{K},$ at each $x\in \mathbb{G}$ is the subspace  $(T_x(x\mathbb{K}))^\perp$ of $T_x\mathbb{G}.$  The linear isometry  defined on horizontal vectors at $x$  by

 $$d\pi_x\vert _{(T_x(x\mathbb{K}))^\perp}:T_x(x\mathbb{K})^\perp\to T_{\pi(x)}\mathbb{G}/\mathbb{K},$$
we shall denote  by $l_x.$

Now consider the map $\Gamma: T(\mathbb{G}/\mathbb{K})\to \mathfrak{g}$  between the tangent bundle  $T(\mathbb{G}/\mathbb{K})$ and the Lie algebra $ \mathfrak{g}$ of $\mathbb{G},$ defined at each $p\in \mathbb{G}/\mathbb{K}$ by 

\begin{eqnarray*}
\Gamma_p: T_p \mathbb{G}/\mathbb{K} &\to &\mathfrak{g}\\
v&\mapsto & (dR_{x^{-1}})_x l_{x}^{-1}(v),
\end{eqnarray*}
where $R$ denotes the right translation  on $\mathbb{G}$ and $x$ is any point in the fiber $\pi^{-1}(p).$ 

\begin{lem}
For each $p\in \mathbb{G}/\mathbb{K},$ the linear transformation $\Gamma_p$ % : T_p(\mathbb{G}/\mathbb{K}) \to \mathfrak{g}$
 is well defined.
\end{lem}
\begin{proof}
If $x, y \in \pi^{-1}(p)$ then $y=R_k(x)$ for some $k\in \mathbb{K}.$ Then, for any horizontal vector $v\in T_x \mathbb{G}$ we observe that $l_x(v)=(d\pi)_y (dR_k)_x  (v)=l_y((dR_k)_x(v)),$ since $R_k$ preserves horizontality. Note that $[(dR_k)_x]^{-1}=(dR_{k^{-1}})_{y} =(dR_{y^{-1}x})_{y}$, then, 

\begin{eqnarray*}
(dR_{x^{-1}})_x l_x^{-1}(v) &=& (dR_{x^{-1}})_x [(dR_k)_x]^{-1}l_y^{-1}(v)\\
 &= & (dR_{x^{-1}})_x (dR_{y^{-1}x})_{y}l_y^{-1}(v)\\
 &= &  (dR_{y^{-1}})_{y}l_y^{-1}(v).
\end{eqnarray*}
Therefore, $\Gamma_p$ is well defined.

\end{proof}

\begin{prop}
For each $v\in \mathfrak{g}$ is possible to associate a Killing vector field $\zeta(v)$ of $N$ such that for all $p\in M$ 
$$\forma{\Gamma_p(u),v}=\forma{u,\zeta(v)(p)},\ \ u\in T_pN.$$
\end{prop}

\begin{proof}
 First note that the left multiplication  of the Lie Group $\mathbb{G}$ on  $\mathbb{G}/\mathbb{K}$ define an isometric action $$\mu:\mathbb{G}\times \mathbb{G}/\mathbb{K} \to  \mathbb{G}/\mathbb{K}$$ given  for each $g\in \mathbb{G}$ by
$$\mu(g, x\mathbb{K})=\pi(L_g(x)), \ \ x\in \mathbb G,$$
 where $L_g$ denotes the left translation by $g$ on $\mathbb{G}.$
  Then  for each $v\in \mathfrak{g}$ the vector field  $\zeta(v) $ of $\mathbb{G}/\mathbb{K}$ determined by 

 $$\zeta(v)(p):=\left.\frac{d}{dt}\mu(\exp(tv),p)\right\vert_{{t=0}},\ \ p\in \mathbb{G}/\mathbb{K},$$ 
  where $\exp$ is the Lie exponential map of $\mathbb{G},$ is a Killing vector field.  Now note that, given $p\in \mathbb{G}/\mathbb{K}$ and $x\in \pi^{-1}(p)$ we have that $\zeta(v)(p)=(d\pi)_x (dR_x)_e (v)$ since

{\footnotesize  $$\zeta(v)(p)=\left.\frac{d}{dt}\mu(\exp(tv),p)\right\vert_{{t=0}}=\left.\frac{d}{dt}\pi(L_{\exp(tv)}(x))\right\vert_{{t=0}}=\left.\frac{d}{dt}\pi(R_x(\exp(tv)))\right\vert_{{t=0}}.$$}
Thus for each $u\in T_p \mathbb{G}/\mathbb{K}$ we then have 
\begin{eqnarray*}
\forma{u,\zeta(v)(p)}&=& \forma{u,(d\pi)_x (dR_x)_e (v)}\\
&=& \forma{l_x^{-1}(u), (dR_x)_e (v)}\\
&=& \forma{(dR_{x^{-1}})_x l_x^{-1}(u), v}\\
&=& \forma{\Gamma_p(u),v}.
\end{eqnarray*}
This concludes the proof the proposition. 

\end{proof}

An important  family of examples of homogenous spaces that can be represented as a quotient as above are the symmetric spaces  \cite{RR}. A simple example is the trivial quotient $\mathbb{G}/\mathbb{K},$ where  $\mathbb G$ is a lie group with a bi-invariant metric and $\mathbb{K}=\{0\}$. The map  $\Gamma:T\mathbb G\to \mathbb R^m:=T_e\mathbb G,$ $m=\dim\mathbb G,$ is given by $\Gamma(p,v)=d(L_p)^{-1}(v)$ where $L_p$ is the left translation on $\mathbb G,$ $L_p(x)=px.$ 

\defin [Gauss map]\label{Gauss}
 Let $M$ be an $n$-dimensional manifold immersed in a  homogeneous space $\mathbb{G}/\mathbb{K},$ such that $\mathbb{G}$ has dimension $m+k$ and $\mathbb{K}$ has dimension $k,$ and assume that $\mathcal{N}_1(M)$ is nonempty. For each $\eta\in \mathcal{N}_1(M)$ we say that the map 
\begin{eqnarray*}
\gamma_\eta:M &\to & \mathbb{S}^{m+k-1}\subset \mathfrak{g}\\
 p & \mapsto & \Gamma_p(\eta(p)).
\end{eqnarray*}
is the \textit{Gauss map associated} to $\eta.$

The Gauss map of hypersurfaces of  connected spaces of constant curvature and of the spaces  $\mathbb{S}^2\times \mathbb{R}$, $\mathbb{H}^2\times \mathbb{R}$ is computed explicitly in \cite{RR} and \cite{LR}, respectively.

In the following result we calculate the Laplacian of  $\gamma_\eta.$

\begin{thm}
Let $M$ be an $n-$\-di\-men\-sio\-nal manifold immersed in a homogeneous space $ \mathbb{G}/\mathbb{K}.$ Then the Gauss map $$\gamma_\eta:M\to \mathbb{S}^{m+k-1}\subset\mathfrak{g}, \ \ \dim\mathbb K=k, \ \ \dim \mathfrak{g}=m+k,$$ associa\-ted to a parallel unit normal vector $\eta$ of $M$ satisfies
 \begin{eqnarray*}
-\Delta\gamma_\eta=\Gamma\left(n\hspace{0,04cm} {\rm grad}\forma{\overrightarrow{H},\eta}+\tilde{\mathcal{B}}(\eta)+{\rm Ric}_M^\perp (\eta)\right)+n\sum_{i=1}^{m+k}\forma{\overrightarrow{H},\n_{\eta} \zeta(v_i)}v_i.
\end{eqnarray*}
where $\{v_1,\ \cdots,\ v_{m+k}\}$ is a fixed orthonormal basis of $\mathfrak{g}.$
\end{thm}

\begin{proof}
Let $\eta$ be a unit parallel normal vector field of $M$ and let $v_1, \cdots,$ $v_{m+k}$ be an orthonormal basis of $\mathfrak{g}.$  Thus, from Corollary~\ref{corol2}

\begin{small}
\begin{eqnarray*}
-\Delta\gamma_\eta &=&-\sum_{i=1}^{m+k}\Delta_M\forma{\gamma_\eta,v_i}v_i\\
&=&- \sum_{i=1}^{m+k}\Delta_M\forma{\eta,\zeta(v_i)}v_i\\
&=& \sum_{i=1}^{m+k}\left( n\forma{{\rm grad}\forma{\overrightarrow{H},\eta},\zeta(v_i)}+\forma{{\rm Ric}_M^\perp(\eta),\zeta(v_i)}+\forma{\tilde{\mathcal{B}}(\eta),\zeta(v_i)}+\right.\\
&& \left. n\forma{\overrightarrow{H},\n_\eta \zeta(v_i)}\right)v_i\\
&=& \Gamma \left( n\hspace{0,04cm} {\rm grad}\forma{\overrightarrow{H},\eta}+{\rm Ric}_M^\perp\eta +\tilde{\mathcal{B}}(\eta)\right)+n\sum_{i=1}^{m+k}\left(\forma{\overrightarrow{H},\n_\eta \zeta(v_i)}\right)v_i.
\end{eqnarray*}
\end{small}
\end{proof}

\begin{corol}\label{Deltamin}
Let $M$ be an $n-$dimensional manifold immersed in an $m-$\-di\-men\-sio\-nal homogeneous space $\mathbb{G}/\mathbb{K}$ and let $\eta$ be a parallel unit normal vector of $M$ such that $\overrightarrow{H}= \norma{\overrightarrow{H}}\eta.$ Then the Gauss map $\gamma_\eta$ associated to $\eta$ satisfies
\begin{eqnarray*}
\Delta\gamma_\eta=-\Gamma\left(n\hspace{0,1cm} {\rm grad}\norma{\overrightarrow{H}}+\tilde{\mathcal{B}}(\eta)+{\rm Ric}_M^\perp (\eta)\right).
\end{eqnarray*}  
In particular, if $M$ is an orientable hypersurface of $N$ then 
\begin{eqnarray}\label{hyperL}
-\Delta\gamma_\eta=n \Gamma({\rm grad}(H))+(\norma{S_\eta}^2+{\rm Ric} (\eta))\gamma_\eta.
\end{eqnarray}  
where ${\rm Ric}$ is the Ricci curvature of $N.$
\end{corol}

Considering the extension $\tilde{\mathcal{B}}+{\rm Ric}_M^\perp $  of the Simons operator to homogeneous spaces,  we  obtain as a direct consequence of Corollary \ref{Deltamin} the following extension, to homogenous spaces,  in the minimal case, of Theorem~\ref{trn}: 

\begin{thm} \label{Homogs}
Let $M$ be an $n-$dimensional manifold immersed  in a $m-$\-di\-men\-sio\-nal homogeneous space $\mathbb{G}/\mathbb{K}.$ Assume that $M$ is minimal and let $\eta$ be a parallel unit normal vector of $M.$   Then, The Gauss map $\gamma_{\eta} $ is harmonic if and only if $\eta$ is an eigenvector of the operator $\tilde{\mathcal{B}}+{\rm Ric}_M^\perp .$ 
Moreover, if $\mathbb{G}/\mathbb{K}$ is a simply connected complete space form $\mathbb{Q}_c^{n}$  with constant sectional curvature $c\in \{-1,0, 1\}.$ We have the equivalences:
\begin{enumerate}
\item $\gamma_{\eta}$ is harmonic.
\item $\eta$ is a harmonic section.
\item $\eta$ is an eigenvector of the operator $\tilde{\mathcal{B}}+{\rm Ric}_M^\perp.$
\end{enumerate}
\end{thm}

The round sphere $\mathbb{S}^n$ can be represented as a quotient $\mathbb G /\mathbb K$ of Lie subgroups $\mathbb G$ and $\mathbb K$ of $O(n+1)$ in many ways, depending on the dimension $n$ (for example, $\mathbb S^n= O(n+1)/O(n)$ for any $n$,  $\mathbb S^n= U(n+1)/U(n)$ for  $n$ odd etc), and the metric of $\mathbb S^n$ is the Riemannian projection of a bi-invariant metric in $\mathbb G.$ In any such representation we can consider the Gauss map as in Definition~\ref{Gauss}.

In the next result  we prove the existence of harmonic unit normal sections on isoparametric submanifolds of the sphere. These sections determine harmonic  Gauss maps, as defined in \ref{Gauss},  for each Lie group quotient representation of $\mathbb S^n$ as described above. 

We prove that if the isoparametric submanifold is also minimal then the associated Gauss maps are eigenmaps of the Laplacian. This result  is an extension of the Theorem~\ref{isorn} and a spherical version of the main result of \cite{SMB}.

 We note that any isoparametric submanifold of $\mathbb{S}^m$ is a leaf of a singular  foliation 
of $\mathbb{S}^m$ by isoparametric submanifolds and any of these foliations contains a leaf which is regular and minimal (see \cite{PT}, pp 138-139).
We have:

\begin{thm}\label{sphere}
Let $M$ be an isoparametric submanifold of $\mathbb{S}^{m}.$ Then:
\begin{enumerate}
\item~\label{i} There exists an orthonormal basis of parallel unit normal sections $\eta_1, \cdots, \eta_r,$ $r=m-dim(M),$ which are eigenvectors of $\tilde{\mathcal{B}}+{\rm Ric}_M^\perp $ with constant eigenvalues. Moreover, the sections $\eta_1, \cdots, \eta_r,$ are harmonic unit normal sections of $M.$

\item If $M$ is a minimal isoparametric submanifold then the Gauss maps 
$$\gamma_{\eta_j} : M \to \mathbb{S}^{m+k-1},$$
 $1 \leq j \leq r$ associated to any
homogeneous representation $\mathbb{G}/\mathbb{K}$ of $\mathbb{S}^m$ are eigenmaps, where $\mathbb{S}^{m+k-1}$ is the
unit sphere of the Lie algebra $\mathfrak{g}$ of $\mathbb{G}.$ 
\end{enumerate}
\end{thm} 

\begin{proof}{\ }
\begin{enumerate}
\item The proof is analogous to the one of Theorem~\ref{isorn} since the isoparametric submanifolds of $\mathbb{S}^{n}$ have globally flat normal bundle. 
\item  It is a direct consequence of Theorem \ref{sphere} since  from the Equation  \eqref{RicciConstantspaces} we have ${\rm Ric}_M^\perp(\eta)=\dim(M)\eta.$
\end{enumerate}
\end{proof}

We can use polar actions (see Section \ref{HUNS}) and Theorem \ref{Polar} to extend Theorem~\ref{sphere} to more general ambient spaces. For example, the  left action
$$\mathbb{K} \times \mathbb{G}/\mathbb{K}$$
of $\mathbb{K}$ on a symmetric space $N= \mathbb{G}/\mathbb{K},$ $k (g\mathbb{K}) = (kg)\mathbb{K}$ (see \cite{AB}). 

In the following theorem, assuming that the Ricci normal operator ${\rm Ric}_M^\perp$ is a multiple of the identity, we obtain harmonic Gauss maps on principal orbits of polar actions.  We note that the normal Ricci curvature ${\rm Ric}_M^\perp$ of submanifold of a space form is always  a multiple of the identity.  If $M$ is a codimension $2$ manifold of an Einstein manifold $N$ then the Ricci normal operator  is given by
$${\rm Ric}_M^\perp(\eta)=({\rm Ric}(\eta) - K(\eta,\eta^\perp))\eta$$
and then is also a multiple of the identity, where  $\eta^\perp$ is a  normal vector to $M$ and orthogonal to $\eta,$ $ K(\eta,\eta^\perp)$  the sectional curvature on $N$ and  ${\rm Ric}(\eta)$ is the Ricci curvature of $N.$  We  mention \cite{HL} where it is classified isometric compact group actions in $\mathbb R^m$ and $\mathbb S^m$ whose principal orbits have codimension $2$.

\begin{thm}
If a principal orbit $M$ of a polar action on a symmetric space $N$ is minimal and the  Ricci normal operator of $M$ is a multiple of the identity then, the associated Gauss maps $\gamma_{\eta_1},\ \cdots, \gamma_{\eta_k}$ of an eigenbasis $\eta_1, \cdots, \eta_k$  of  $\tilde{\mathcal{B}},$ as in  Theorem~\ref{Polar},   are harmonic maps. 
\end{thm}

We note that if an orbit of an action of a compact Lie group on a manifold $N$ is a local maximum for the volume then this orbit is a critical point of the area functional and hence is minimal (see \cite{HL}). Therefore, if $N$ is compact, there always exists an orbit which is minimal. 

\subsection{The sphere $\mathbb{S}^7$ with the octonionic multiplication}\label{s7}

In this section we apply the computations of  Section \ref{rl} to  constructions of \cite{SML}. It allows us to obtain a version of  Ruh-Vilms Theorem of hypersurfaces of  $\mathbb S^k$ with $3\leq k\leq 7,$ where the Gauss map depends of the octonionic structure of $\mathbb{R}^8.$  The results in this section extends Corollary 1.2 and Theorem 1.4 of  \cite{SML}.  We begin recalling some basic facts about octonionic geometry. 

Given $n\in\{0,1,2,\dots\},$ the Cayley-Dickson algebra
$\mathcal{C}_{n}$ is a division algebra structure on $\mathbb{R}^{2^{n}}$
defined inductively by $\mathcal{C}_{0}=\mathbb{R}$ and by the following
formulae: If $x=\left(  x_{1},x_{2}\right)  $, $y=\left(  y_{1},y_{2}\right)
$ are in $\mathbb{R}^{2^{n}}=\mathbb{R}^{2^{n-1}}\times\mathbb{R}^{2^{n-1}}$,
$n\geq1$, then
\begin{equation}
x\cdot y=\left(  x_{1}y_{1}-\overline{y_{2}}x_{2},y_{2}x_{1}+x_{2}%
\overline{y_{1}}\right)  , \label{xy}%
\end{equation}
where
\[
\overline{x}=\left(  \overline{x}_{1},-x_{2}\right)  ,
\]
with $\overline{x}=x$ if $x\in\mathbb{R}$ (see \cite{B}). 

 The Cayley-Dickson algebra $\mathcal{C}_3$ of $\mathbb R^8$ is called the octonions and denoted by $\mathbb{O}.$
 We next mention some well known
facts about the octonions which proofs can be found in \cite{B}. Let $1$ be the neutral element of $\mathbb{O}.$ Besides being
a division algebra, $\mathbb{O}$ is normed: $\Vert x\cdot y \Vert=\Vert x
\Vert\Vert y \Vert,$ for any $x,\ y\in\mathbb{O}$, where $\Vert\ \Vert$ is the
usual norm of $\mathbb{R}^{8},$ and $\Vert x \Vert=\sqrt{x\cdot\overline{x}}.$
Setting $\operatorname*{Re}(x)=\left(  x+\overline{x}\right)  /2$ we have
\[
T_{1}\mathbb{S}^{7}=\{x\in\mathbb{R}^{8}\ | \ \operatorname*{Re}(x)=0\}.
\]

The right and left translations $R_{x}, L_{x}:\mathbb{O}\rightarrow
\mathbb{O},$ $R_{x}(v)=v\cdot x,$ $L_{x}(v)=x\cdot v,$ $v\in\mathbb{O},$ are
orthogonal maps if $\Vert x \Vert=1$ and are skew-symmetric if
$\operatorname*{Re}(x)=0.$ In particular, the unit sphere $\mathbb{S}^{7}$ centered at the origin of $\mathbb R^8$ is
preserved by left and right translation of unit vectors and, moreover, any
$v\in T_{1}\mathbb{S}^{7}$ determines a Killing vector field $V$ of
$\mathbb{S}^{7}$ given by the left translation, $V(x)=x\cdot v$,
$x\in\mathbb{S}^{7}.$

Defining a translation on $\mathbb{S}^7$   by
\begin{align*}
\Gamma: T\mathbb{S}^{7}  &  \to T_{1}\mathbb{S}^{7}\\
(x,v)  &  \mapsto L_{x^{-1}}(v),
\end{align*} 
it follows immediatly from the properties of the octonions described above that $\Gamma$ is a linear transformation and the Killing vector field $V$ of $\mathbb{S}^7$ determined by $v\in T_1\mathbb{S}^7$ by left multiplication satisfies 
$$\forma{\Gamma(x,u),v}=\forma{u,V(x)},\ \ u\in T_xN, \ \ x\in \mathbb{S}^7.$$

Herein, we identify a $k-$dimensional unit sphere $\mathbb{S}^k,$ $k=3,\ 4 ,\ 5,\ 6,$ as a totally geodesic sphere of $\mathbb{S}^7.$  Then, we may define \textit{the octonionic Gauss map} of an oriented hypersurface $M$ of  $\mathbb{S}^k,$ with $3\leq k\leq 7.$ by 
\begin{eqnarray*}
\gamma_\eta: M &\to & \mathbb{S}^6\subset T_1\mathbb{S}^7\\
x &\mapsto & x^{-1}\cdot \eta,
\end{eqnarray*} 
where $\eta$ is a unit normal vector of $M$ in $\mathbb{S}^k.$

Regarding  $M$ as a submanifold of the sphere $\mathbb{S}^7$ it is easy to see that $\eta$ is a parallel unit normal vector $\eta$ and an eigenvector of $\tilde{\mathcal{B}}+{\rm Ric^\perp_M}.$ Therefore, as a consequence of Corollary~\ref{corol2} we have the following version of the Ruh-Vilms Theorem:
 
  \begin{thm}\label{RVSk}
 Let $M$ be an oriented immersed hypersurface of $\mathbb{S}^k,$  $3\leq k\leq 7.$  Then, the octonionic Gauss map of $M$ in $\mathbb{S}^k$ is harmonic if and only if $M$ has CMC.
 \end{thm}
 \begin{proof}
 Let $M$ be an oriented immersed hypersurface of $\mathbb{S}^k.$  If $\eta$ is a unit normal vector of $M$ in $\mathbb{S}^k$ then $\eta$ is a parallel vector in the normal bundle  of $M$ in $\mathbb{S}^7$ and 
 $$\left(\tilde{\mathcal{B}}+{\rm Ric^\perp_M}\right)(\eta)=\left(\norma{B}^2+(k-1)\right)\eta, $$
where $B$ is the second fundamental form of $M$ in $\mathbb{S}^k.$ Then, from Co\-rollary~\ref{corol2} we have 
\begin{eqnarray}\label{lapoc}
-\Delta\gamma_\eta=(k-1)\Gamma({\rm grad}(H))+\left(\norma{B}^2+(k-1)\right)\gamma_\eta,
\end{eqnarray}
where  $H$ is the mean curvature of $M$ in $\mathbb{S}^k.$ Therefore, $\gamma_\eta$ is a harmonic map if and only if $H$ is constant. 

 \end{proof}
 
A Gauss map of an orientable hypersurface $M$ of $\mathbb{S}^{n},$
$n\geq3,$ that have been often studied in the literature associates, to each
$x\in M,$ the vector $\eta(x)\in\mathbb{S}^{n},$ where $\eta$ is a unit normal
vector field along $M.$ E. De Giorgi \cite{DG} and J. Simons \cite{JS} proved that if $M$ is compact, has constant mean curvature and the image of the Gauss map is
contained in an \emph{open} hemisphere of $\mathbb{S}^{n}$ then $M$ must be a
totally geodesic hypersphere of $\mathbb{S}^{n}.$ Using the octonionic Gauss
map we obtain:

\begin{thm}
Let $M$ be a compact and oriented hypersurface  of $\mathbb{S}^k,$ $3\leq k\leq 7.$  If $M$ has CMC then the image of the octonionic Gauss map of $M$ is not contained in an open hemisphere of $\mathbb{S}^6\subset T_1\mathbb{S}^7.$
\end{thm} 

\begin{proof}
Assume that the image of $\gamma:=\gamma_\eta$ is contained in an open hemisphere of $\mathbb{S}^{6}$
having $v\in\mathbb{S}^{6}$ as pole. This means that $\left\langle
\gamma(x),v\right\rangle >0$ for all $x\in M.$ We claim that one may
find $7$ linearly independent vectors $v_{1},\cdots,v_{7}\in T_{1}\mathbb{S}^{7}$
such that $\left\langle \gamma(x),v_{i}\right\rangle >0,$ for all $x\in M$ and
$1\leq i\leq7.$ Indeed: 
Note that%
\[
\left\langle \gamma\left(  x\right)  ,v\right\rangle >0\Leftrightarrow
d\left(  \gamma\left(  x\right)  ,v\right)  <\frac{\pi}{2}%
\]
for all $x\in M,$ where $d$ is the distance in $\mathbb{S}^{6}.$ By
compactness there is $\varepsilon>0$ such that%
\[
d\left(  \gamma\left(  x\right)  ,v\right)  <\varepsilon<\frac{\pi}{2}%
\]
for all $x\in M.$ Let $U$ be a neighourhood of $v$ in $\mathbb{S}^{6}$ such
that
\[
d\left(  v,u\right)  <\frac{\pi}{2}-\varepsilon
\]
for all $u\in U.$ Then%
\[
d\left(  \gamma\left(  x\right)  ,u\right)  \leq d\left(  \gamma\left(
x\right)  ,v\right)  +d\left(  v,u\right)  <\frac{\pi}{2}%
\]
that is, the inequality
\[
\left\langle \gamma\left(  x\right)  ,u\right\rangle >0
\]
holds for all $x\in M$ and all $u\in U.$ Since $\mathbb{R}^{7}$ is the
smallest linear subspace of $\mathbb{R}^{7}$ containing $U$ the claim is proved.

From \eqref{lapoc} we obtain
\begin{eqnarray}\label{suph}
\Delta_M \left\langle \gamma,v_{i}\right\rangle=-(\norma{B}^2+(k-1))\forma{\gamma,v_i} <0. 
\end{eqnarray} 
 The functions $\left\langle \gamma,v_{i}\right\rangle $ being superharmonic and not changing sign must be constant. Hence $\gamma$ must be a constant map and then $\Delta\gamma=0.$ From \eqref{suph} $\gamma$ must be
zero, contradiction! This proves the theorem.
\end{proof}

\end{document}